\documentclass[11pt]{article}
\usepackage[latin1]{inputenc}
\usepackage{epsfig}
\usepackage{color}
\usepackage[british,english]{babel}
\usepackage{amsthm}
\usepackage{amsmath}
\usepackage{amsfonts}
\usepackage{amssymb}
\usepackage{graphicx}
\usepackage{fancyhdr}

\newtheorem{corollary}{Corollary}[section]

\newtheorem{lemma}[corollary]{Lemma}

\newtheorem{remark}[corollary]{Remark}
\newtheorem{theorem}[corollary]{Theorem}
\newfont{\sBlackboard}{msbm10 scaled 900}

\newcommand{\mylabel}[1]{\label{#1}
            \ifx\undefined\stillediting
            \else \fbox{$#1$}\fi }
\newcommand{\BE}{\begin{equation}}

\newcommand{\EEQ}{\end{equation}}
\newcommand{\rfb}[1]{\mbox{\rm
   (\ref{#1})}\ifx\undefined\stillediting\else:\fbox{$#1$}\fi}

\newfont{\Blackboard}{msbm10 scaled 1200}

\newfont{\roma}{cmr10 scaled 1200}

\def\CC{\rm \hbox{C\kern-.56em\raise.4ex
         \hbox{$\scriptscriptstyle |$}\kern+0.5 em }}

\def\n{|\kern -.05cm{|}\kern -.05cm{|}}

\newcommand{\mm}    {{\hbox{\hskip 0.5pt}}}

\newcommand{\bluff} {{\hbox{\raise 15pt \hbox{\mm}}}}
\makeatletter
\def\section{\@startsection {section}{1}{\z@}{-3.5ex plus -1ex minus
    -.2ex}{2.3ex plus .2ex}{\large\bf}}
\makeatother
\def\be{\begin{equation}}
\def\ee{\end{equation}}

\date{ }
\begin{document}
\thispagestyle{empty}
\title{\bf Nonlinear Reynolds equations for non-Newtonian thin-film fluid flows over a rough boundary}\maketitle

\author{ \center  Mar\'ia ANGUIANO\\
Departamento de An\'alisis Matem\'atico. Facultad de Matem\'aticas\\
Universidad de Sevilla, P. O. Box 1160, 41080-Sevilla (Spain)\\
anguiano@us.es\\}
\medskip\author{ \center  Francisco Javier SU\'AREZ-GRAU\\ Departamento de Ecuaciones Diferenciales y An\'alisis Num\'erico. Facultad de Matem\'aticas \\ Universidad de Sevilla, 41012-Sevilla (Spain)\\
 fjsgrau@us.es\\}

\vskip20pt

 \renewcommand{\abstractname} {\bf Abstract}
\begin{abstract} 
We consider a non-Newtonian fluid flow in a thin domain with thickness $\eta_\varepsilon$ and an oscillating top boundary of period $\varepsilon$. The flow is described by the 3D incompressible Navier-Stokes system with a nonlinear viscosity, being a power of the shear rate (power law) of flow index $p$, with $9/5\leq p<+\infty$.  We consider the limit when the thickness tends to zero and we prove that the three characteristic regimes for Newtonian fluids are still valid for non-Newtonian fluids, i.e. Stokes roughness ($\eta_\varepsilon\approx \varepsilon$), Reynolds roughness ($\eta_\varepsilon\ll \varepsilon$) and high-frequency roughness ($\eta_\varepsilon\gg \varepsilon$) regime.
Moreover, we obtain different nonlinear Reynolds type equations in each case.
\end{abstract}
\bigskip\noindent

 {\small \bf AMS classification numbers:} 76D08, 76A20, 76A05, 76M50, 35Q30.  \\
 
   \newpage

\section {Introduction}\label{S1}
The classical lubrication problem is to describe the situation in which two adjacent surfaces in relative motion are separated by a thin film of fluid acting as a lubricant. Such situation appears naturally in numerous industrial and engineering applications, in particular those consisting of moving machine parts. The mathematical models for describing the motion of the lubricant usually result from the simplification of the geometry of the lubricant film, i.e. its thickness. Using the film thickness as a small parameter, an asymptotic approximation of the Stokes system can be derived providing the well-known Reynolds equation for the pressure of the fluid (see Bayada and Chambat \cite{Bayada1} or Reynolds \cite{Reynolds} for more details). For the stationary case,  considering no-slip condition on the boundary and an exterior force $\tilde f'$,   the two-dimensional Reynolds equation for the unknown pressure $\tilde p$ has the form
\begin{equation}\label{classical_Rey}
{\rm div}_{x'}\left({h(x')^3\over 12 \mu}\left(\tilde f'(x')-\nabla_{x'} \tilde p(x')\right)\right)=0\,,
\end{equation}
where $h$ describe the shape of the top boundary   and $\mu$ is the fluid viscosity.

Engineering practice also stresses the interest of studying the effects of  domain irregularities on a thin film flow.  Thus, the goal becomes in identifying in which way the irregular boundary affects the flow. In this sense, the oscillating boundary is described by two parameters, $\varepsilon$ and $\eta_\varepsilon$, which are devoted to tend to zero. The parameter $\varepsilon$ is the characteristic  wavelength of the periodic roughness, and $\eta_\varepsilon$ is the thickness of the domain, i.e. the distance between the surfaces.  By means of homogenization thecniques, it is showed in Bayada and Chambat \cite{Bayada_Chambat_2,Bayada_Chambat} that depending in the critical size, $\eta_\varepsilon\approx \varepsilon$ with $\eta_\varepsilon/\varepsilon\to \lambda$, $0<\lambda<+\infty$, there exist three types of flow regimes. This result has been successfully generalized to the unstationary case (the rough surface is moving) in Fabricius {\it et al.} \cite{Fab1,Fab2}. Below, we describe the three characteristic regimes:
\begin{itemize}
\item[$\bullet$ ] Stokes roughness regime: it corresponds to the critical case when the thickness of the domain is proportional to the wavelength of the roughness, with $\lambda$ the proportionality constant, $0<\lambda<+\infty$ (see Figure \ref{fig: Stokes_roughness_regime}). In this case, a modified Reynolds equation is obtained as an effective model where the coefficients are obtained by solving 3D local Stokes problems which depend on the parameter $\lambda$.
\item[$\bullet$ ] Reynolds roughness regime: it corresponds to the case when $\lambda=0$, i.e. $\eta_\varepsilon\ll \varepsilon$ and so the wavelength of the roughness is much greater than the film thickness (see Figure \ref{fig: Reynolds_roughness_regime}). In this case, a modified Reynolds equation is obtained  as an effective model where the coefficients are obtained by solving 2D local Reynolds problems. Similar averaged effective equations appear for example in  \cite{Patir,Phan,Wall}.
\item[$\bullet$] High-frequency regime:  it corresponds to the case when $\lambda=+\infty$, i.e. $\eta_\varepsilon\gg \varepsilon$ and so the wavelength of the roughness is much smaller than the film thickness (see Figure \ref{fig:High_frequency_regime}). In this case, due to the highly oscillating boundary, the velocity field vanishes in the oscillating zone and a simpler Reynolds equation is deduced in the non-oscillating zone.\end{itemize}
\begin{figure}[h!]
\begin{center}
\includegraphics[width=10.5cm]{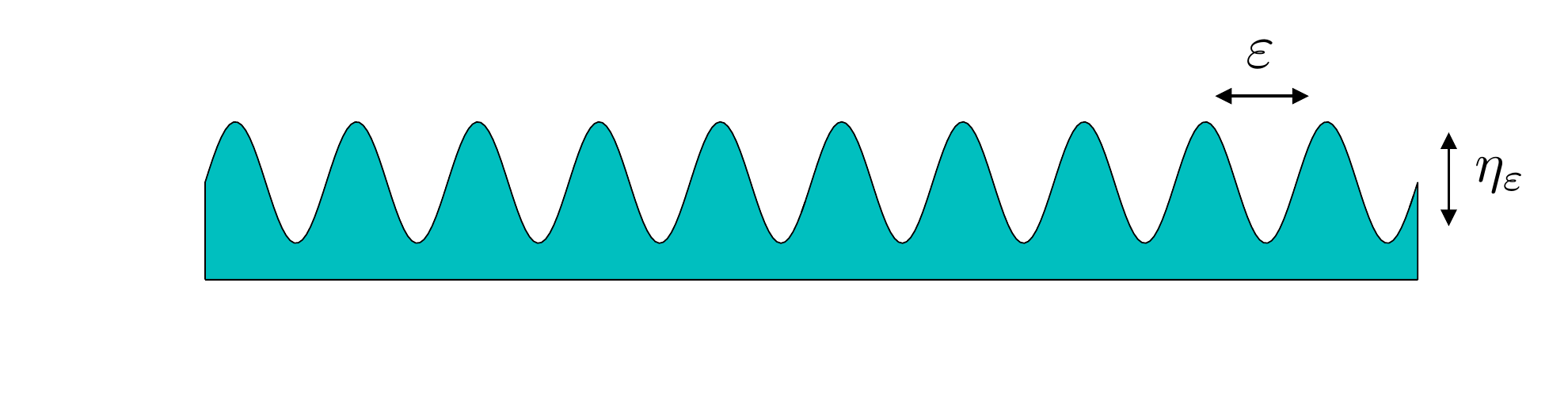}
 \vspace{-1cm}
 \end{center}
\caption{Stokes roughness regime}
\label{fig: Stokes_roughness_regime}
\end{figure}

\begin{figure}[h!]
\begin{center}
\includegraphics[width=10.5cm]{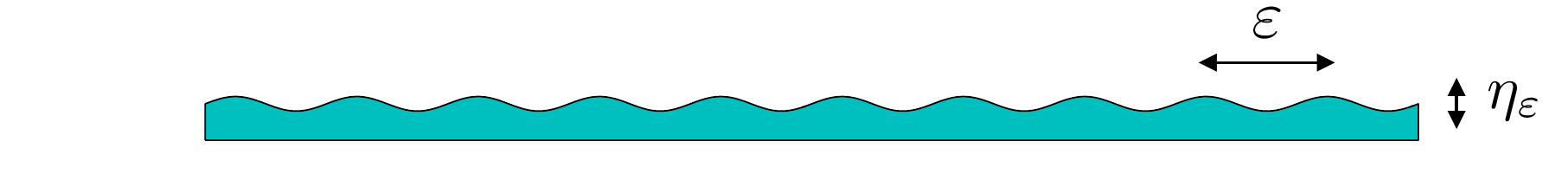}
 \vspace{-1cm}
 \end{center}
\caption{Reynolds roughness regime}
\label{fig: Reynolds_roughness_regime}
\end{figure}

\begin{figure}[h!]
\begin{center}
\includegraphics[width=10.5cm]{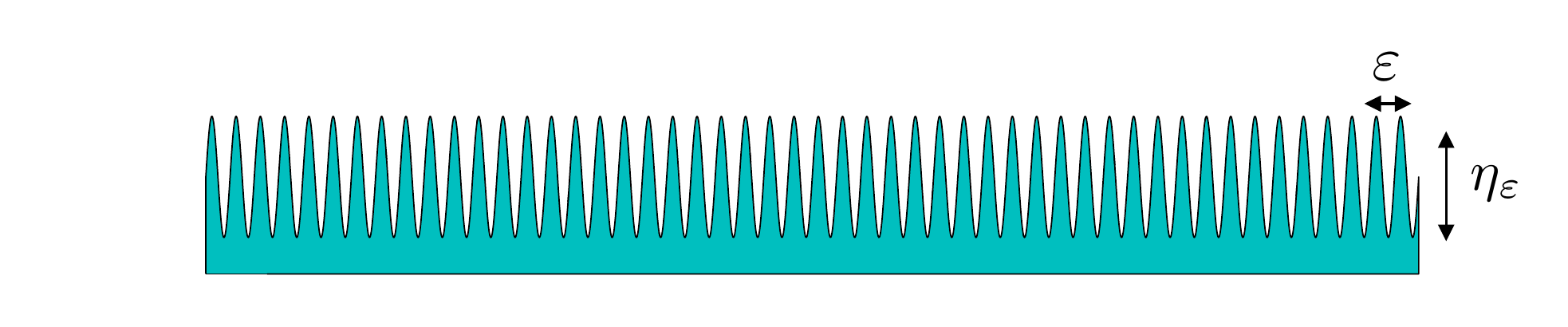}
 \vspace{-1cm}
 \end{center}
\caption{High-frequency regime}
\label{fig:High_frequency_regime}
\end{figure}

This problem is well studied in the case of Newtonian fluids, however, for the non-Newtonian fluids the situation is completely different.  The main reason is that the viscosity is a nonlinear function of the symmetrized gradient of the velocity. A relevant case of non-Newtonian fluids is when the viscosity satisfies the nonlinear power law, which is widely used for melted polymers, oil, mud, etc. If $u$ is the velocity and $Du$ the gradient velocity tensor, denoting the shear rate by $\mathbb{D}\left[u\right]=\frac{1}{2}(Du+D^t u)$, the viscosity as a function of the shear rate is given by $$\eta_{p} \left( \mathbb{D}\left[u\right]\right)=\mu \left\vert \mathbb{D}\left[u\right] \right\vert^{p-2},\text{ \ \ }1<p<+\infty,$$where the two material parameters $\mu>0$ and $p$ are called the consistency and the flow index, respectively. Recall that $p=2$ yields the Newtonian fluid, for $1<p<2$ the fluid is pseudoplastic (shear thinning), which is the characteristic of high polymers, polymer solutions, and many suspensions, whereas for $p>2$ the fluid is dilatant (shear thickening), whose behavior is reported for certain slurries, like mud, clay or cement, and implies an increased resitence to flow with intesified shearing.

Similarly to the mathematical derivation of the 2D Reynolds equation (\ref{classical_Rey}) for Newtonian fluids, a 2D nonlinear Reynolds equation  for non-Newtonian fluids has been obtained in Bourgeat {\it et al.} \cite{Bourgeat} and Mikeli\'c  and Tapiero \cite{MT}, which has the form

$${\rm div}_{x'}\left({h(x')^{p'+1}\over 2^{p'\over 2}(p'+1)\mu^{p'-1}}\left|\tilde f'(x')-\nabla_{x'}\tilde p(x')\right|^{p'-2}\left(\tilde f'(x')-\nabla_{x'}\tilde p(x')\right)\right)=0\,,$$
where $p^\prime=p/(p-1)$ is the conjugate exponent of $p$. \\

In this paper, we consider fluid flows satisfying the non-Newtonian Navier-Stokes system, where the viscosity satisfies the nonlinear power law with $9/5\leq p< +\infty$, in the thin domain with a rough boundary  described above (see Fig. \ref{fig: Stokes_roughness_regime}, Fig. \ref{fig: Reynolds_roughness_regime} and Fig. \ref{fig:High_frequency_regime}). Our purpose is to study the asymptotic behavior of this system when $\varepsilon$ and $\eta_\varepsilon$ tend to zero. The proof of our results is based on an adaptation of the unfolding method (see Arbogast {\it et al.} \cite{arbogast}, and Cioranescu {\it et al.} \cite{Ciora}), which is strongly related to the two-scale convergence method (see Allaire \cite{Allaire1}, and  Nguetseng \cite{Nghe}), but here it is necessary to combine it with a rescaling in the height variable, in order to work with a domain of fixed height, and to use monotonicity arguments to pass to the limit. The unfolding method is a very efficient tool to study periodic homogenization problems where the size of the periodic cell tends to zero. The idea is to introduce suitable changes of variables which transform every periodic cell into a simpler reference set by using a supplementary variable (microscopic variable). Thanks to this method, we are able to identify the critical size and later the effects of the microstructure in the corresponding effective equations. Thus, we obtain that the critical size is exactly the same as the one of the Newtonian case, i.e. when $\eta_\varepsilon\approx \varepsilon$ with $\eta_\varepsilon/\varepsilon\to \lambda$, $0<\lambda<+\infty$. This means that the same three characteristic regimes are still valid for the non-Newtonian case: the Stokes roughness regime ($\eta_\varepsilon\approx\varepsilon$), the Reynolds roughness regime ($\eta_\varepsilon\ll\varepsilon$) and the high-frequency regime ($\eta_\varepsilon\gg\varepsilon$). As a result, we generalize the Newtonian case studied by Bayada and Chambat \cite{Bayada_Chambat_2,Bayada_Chambat} to the case of a non-Newtonian fluid governed by the Navier-Stokes system and we give the explicit expressions in each regime, which are the main novelties of the paper.

 Some other generalized nonlinear Reynolds equations for non-Newtonian fluids has been also obtained in Duvnjak \cite{Duv} for lubrication of a rotating shaft,  in Boukrouche {\it et al.} \cite{Bouk1} and Boukrouche and El Mir \cite{Bouk2}, where it is assumed stick-slip conditions given by Tresca law on the boundary, and   in Su\'arez-Grau \cite{grau1}, where  Navier slip boundary conditions are prescribed on the rough boundary.\\

The plan of this paper is as follows. In Section \ref{S2}, the domain and some notations are introduced. In Section \ref{S3}, we formulate the problem and state our main result, which is proved in Section \ref{S6} using {\it a priori} estimates and compactness results established in Section \ref{S4} and Section \ref{S5}, respectively.

\section{The domain and some notations}\label{S2}
Along this section, the points $x\in\mathbb{R}^3$ will be decomposed as $x=(x^{\prime},x_3)$ with $x^{\prime}\in \mathbb{R}^2$, $x_3\in \mathbb{R}$. We also use the notation $x^{\prime}$ to denote a generic vector of $\mathbb{R}^2$.

We consider a smooth bounded open set $\omega\subset \mathbb{R}^2$. The thin domain with an oscillating boundary is defined by
\begin{equation}\label{Dominio1}\Omega_\varepsilon=\left\{x\in\mathbb{R}^3\,:\, x'\in\omega,\ 0<x_3< \eta_\varepsilon\, h\left({x'\over \varepsilon}\right)\right\}\,,
\end{equation}
where the oscillating part of the boundary $\partial \Omega_\varepsilon$ is given by
\begin{equation*}\label{Oscbound_1}
\Sigma_\varepsilon=\left\{x\in\mathbb{R}^3\,:\, x'\in\omega,\  x_3= \eta_\varepsilon\, h\left({x'\over \varepsilon}\right)\right\}\,.
\end{equation*}
Here, $\eta_\varepsilon h(x'/\varepsilon)$ represents the real gap between the two surfaces and $h$ is a smooth function, defined for $y'$ in $\mathbb{R}^2$, $Y'$-periodic,  being $Y'=(-1/2,1/2)^2$ the cell of periodicity. The small parameter $\eta_\varepsilon$ is related to the film thickness, whereas the small parameter $\varepsilon$ is the wavelength of the roughness.
 
In order to have a domain with thickness order one, we use the dilatation in the variable $x_3$ given by
\begin{equation}\label{dilatacion}
y_3=\frac{x_3}{\eta_\varepsilon}\,,
\end{equation}
which transforms the thin domain $\Omega_\varepsilon$ in the rescaled domain $\widetilde \Omega_\varepsilon$ given by
\begin{equation}\label{Dominio2}\widetilde \Omega_\varepsilon=\left\{(x',y_3)\in\mathbb{R}^2\times \mathbb{R}\,:\, x'\in\omega,\ 0<y_3<  h\left({x'\over \varepsilon}\right)\right\}\,,
\end{equation}
where the oscillating part of the boundary $\partial \widetilde \Omega_\varepsilon$ is given by
\begin{equation*}\label{Oscbound_2}
\widetilde \Sigma_\varepsilon=\left\{(x',y_3)\in\mathbb{R}^2\times \mathbb{R}\,:\, x'\in\omega,\  y_3 = h\left({x'\over \varepsilon}\right)\right\}\,.
\end{equation*}
We denote $$h_{\rm min}=\min_{y'\in Y'}h(y'),\quad h_{\rm max}=\max_{y'\in Y'}h(y')\,,
$$
and we define the domain with a fixed height $\Omega$ by
 \begin{equation*}\label{Omega}
 \Omega=\{(x',y_3)\in \mathbb{R}^2\times \mathbb{R}\,:\, x'\in\omega,\ 0<y_3<h_{\rm max}\}\,,
 \end{equation*}
and the corresponding top boundary  $\Sigma$ by
\begin{equation*}\label{Sigma_top}
 \Sigma=\{(x',y_3)\in \mathbb{R}^2\times \mathbb{R}\,:\, x'\in\omega,\ y_3=h_{\rm max}\}\,.
 \end{equation*}
We also define 
\begin{equation*}\label{Omega_menos}
 \Omega^-=\{(x',y_3)\in \mathbb{R}^2\times \mathbb{R}\,:\, x'\in\omega,\ 0<y_3<h_{\rm min}\}\,.
 \end{equation*}

We denote by $Y$ the reference cell in $\mathbb{R}^3$, which is given by
\begin{equation}\label{cell}Y=\{ y\in \mathbb{R}^3\,:\, y'\in Y',\ 0<y_3< h(y')\}\,,
\end{equation}
and by $L^p_{\sharp} (Y)$, $W^{1,p}_{\sharp}(Y)$, with $1<p<+\infty$, the functional spaces
$$\begin{array}{l}\displaystyle
L^p_{\sharp} (Y)=\Big\{ v\in L^p_{loc}(Y)\, : \,\int_{Y}
|v|^p dy <+\infty, \\
\qquad\qquad\quad
 v(y'+k',y_3)=v(y)\hskip 0.2cm \forall k'\in\mathbb{Z}^2,\, \mbox{a.e. }y\in Y
\Big\},\end{array}
$$
and
$$\begin{array}{l}\displaystyle
W^{1,p}_{\sharp} (Y)= \Big\{ v \in W^{1,p}_{loc}(Y)\cap L^p_{\sharp} (Y):  \int_{Y}
|\nabla_{y}v|^p dy <+\infty\Big\}.
\end{array}
$$

We denote by $O_\varepsilon$ a generic real sequence which tends to zero with $\varepsilon$ and can change from line to line. We denote by $C$ a generic positive constant which can change from line to line.
\section{Setting and main results}\label{S3}
In this section we describe the asymptotic behavior of an incompressible viscous non-Newtonian fluid in the geometry $\Omega_{\varepsilon}$ given by (\ref{Dominio1}). The proof of the corresponding results will be given in the next sections.

Our results are referred to the stationary non-Newtonian Navier-Stokes system, \begin{equation}
\left\{
\begin{array}
[c]{r@{\;}c@{\;}ll}%
\displaystyle -{\rm div}\, \left(\eta_{p} \left( \mathbb{D}\left[u_\varepsilon\right]\right)\mathbb{D}\left[u_\varepsilon\right] \right)+(u_\varepsilon\cdot \nabla)u_\varepsilon+ \nabla p_{\varepsilon} &
= &
f \text{\ in \ }\Omega_{\varepsilon},\\
{\rm div}\, u_{\varepsilon} & = & 0 \text{\ in \ }\Omega_{\varepsilon},
\end{array}
\right. \label{1}%
\end{equation}
where $u_\varepsilon$ is the velocity, $p_\varepsilon$ is the pressure (scalar) and $p^\prime=p/(p-1)$ is the conjugate exponent of $p$. The right-hand side $f$ is of the form $$f(x)=(f^{\prime}(x^{\prime}),0), \text{\ a.e. \ }x\in \Omega,$$
where $f$ is assumed in $L^{p^\prime}( \omega\times (-h_{\rm max},h_{\rm max}))^2$. This choice of $f$ is usual when we deal with thin domains. Since the thickness of the domain, $\eta_\varepsilon$, is small then the vertical component of the force can be neglected and, moreover the force can be considered independent of the vertical variable. 

Finally, we may consider no-slip boundary conditions without altering the generality of the problem
under consideration, 
\begin{equation}\label{CF}
u_{\varepsilon}=0 \text{\ on \ } \partial \Omega_{\varepsilon}.
\end{equation}
It is well known that (\ref{1})-(\ref{CF}) admits at least one weak solution  $(u_{\varepsilon},p_{\varepsilon})\in W_0^{1,p}(\Omega_{\varepsilon})^3\times L^{p^\prime}_0(\Omega_{\varepsilon})$ with $9/5\leq p<+\infty$ (see Lions \cite{Lions2} and M\'alek {\it et al.} \cite{Nec} for more details).  The space $L^{p^\prime}_0(\Omega_{\varepsilon})$ is the space of functions of $L^{p'}(\Omega_\varepsilon)$ with null integral.

Our aim is to study the asymptotic behavior of $u_{\varepsilon}$ and $p_{\varepsilon}$ when $\varepsilon$ and $\eta_\varepsilon$ tend to zero. For this purpose, as usual when we deal with thin domains, we use the dilatation in the variable $x_3$ given by (\ref{dilatacion}) in order to have the functions defined in the open set  $\widetilde\Omega_\varepsilon$ defined by (\ref{Dominio2}).

Namely, we define $\tilde{u}_{\varepsilon}\in W_0^{1,p}(\widetilde{\Omega}_{\varepsilon})^3$, $\tilde{p}_{\varepsilon}\in L^{p^\prime}_0(\widetilde{\Omega}_{\varepsilon})$ by $$\tilde{u}_{\varepsilon}(x^{\prime},y_3)=u_{\varepsilon}(x^{\prime},\eta_\varepsilon y_3),\text{\ \ }\tilde{p}_{\varepsilon}(x^{\prime},y_3)=p_{\varepsilon}(x^{\prime},\eta_\varepsilon y_3), \text{\ \ } a.e.\text{\ } (x^{\prime},y_3)\in \widetilde{\Omega}_{\varepsilon}.$$
Let us introduce some notation which will be useful in the following. For a vectorial function $v=(v',v_3)$ and a scalar function $w$, we will denote $\mathbb{D}_{x^\prime}\left[v\right]=\frac{1}{2}(D_{x^\prime}v+D_{x^\prime}^t v)$ and $\partial_{y_3}\left[v\right]=\frac{1}{2}(\partial_{y_3}v+\partial_{y_3}^t v)$, where we denote $\partial_{y_3}=(0,0,\frac{\partial}{\partial y_3})^t$, and associated to the change of variables (\ref{dilatacion}), we introduce the operators: $\mathbb{D}_{\eta_\varepsilon}$, $D_{\eta_\varepsilon}$, ${\rm div}_{\eta_\varepsilon}$ and $\nabla_{\eta_\varepsilon}$ by
\begin{equation*}
\mathbb{D}_{\eta_\varepsilon}\left[v\right]=\frac{1}{2}\left(D_{\eta_\varepsilon} v+D^t_{\eta_\varepsilon} v \right),\quad {\rm div}_{\eta_\varepsilon}v={\rm div}_{x^{\prime}}v^{\prime}+\frac{1}{\eta_\varepsilon}\partial_{y_3}v_3\,,
\end{equation*}
\begin{equation*}\begin{array}{l}
(D_{\eta_\varepsilon}v)_{i,j}=\partial_{x_j}v_i\text{\ for \ }i=1,2,3,\ j=1,2,\\
\\
\displaystyle (D_{\eta_\varepsilon}v)_{i,3}=\frac{1}{\eta_\varepsilon}\partial_{y_3}v_i\text{\ for \ }i=1,2,3\,,
\end{array}
\end{equation*} 
\begin{equation*}
\nabla_{\eta_\varepsilon}=(\nabla_{x'} w,{1\over \eta_\varepsilon}\partial_{y_3} w)^t\,.
\end{equation*}

Using the transformation (\ref{dilatacion}), the system (\ref{1}) can be rewritten as
\begin{equation}
\left\{
\begin{array}
[c]{r@{\;}c@{\;}ll}%
\displaystyle \!\!\!\!-{\rm div}_{\eta_\varepsilon} \left( \!\mu \left\vert \mathbb{D}_{\eta_\varepsilon}\left[\tilde{u}_\varepsilon\right] \right\vert^{p-2}\mathbb{D}_{\eta_\varepsilon}\left[\tilde{u}_\varepsilon\right] \right)\!+\! (\tilde u_\varepsilon\cdot \nabla_{\eta_\varepsilon})\tilde u_\varepsilon\!+\! \nabla_{\eta_\varepsilon} \tilde{p}_{\varepsilon} &
= &
f \text{\ in \ }\widetilde{\Omega}_{\varepsilon},\\
{\rm div}_{\eta_\varepsilon} \tilde{u}_{\varepsilon} & = & 0 \text{\ in \ }\widetilde{\Omega}_{\varepsilon},
\end{array}
\right. \label{2}%
\end{equation}
with no-slip  condition, i.e. 
\begin{equation}\label{no-slip_tilde}
\tilde{u}_{\varepsilon}=0 \text{\ on \ } \partial \widetilde{\Omega}_{\varepsilon}.
\end{equation}
Our goal then is to describe the asymptotic behavior of this new sequence $(\tilde{u}_{\varepsilon}$, $\tilde{p}_{\varepsilon})$. 

The sequence of solutions $(\tilde{u}_{\varepsilon}$, $\tilde{p}_{\varepsilon})\in W_0^{1,p}(\widetilde{\Omega}_{\varepsilon})^3 \times  L^{p^\prime}_0(\widetilde{\Omega}_{\varepsilon})$ is not defined in a fixed domain independent of $\varepsilon$ but rather in a varying set $\widetilde{\Omega}_{\varepsilon}$. In order to pass the limit if $\varepsilon$ tends to zero, convergences in fixed Sobolev spaces (defined in $\Omega$) are used which requires first that $(\tilde{u}_{\varepsilon}$, $\tilde{p}_{\varepsilon})$ be extended to the whole domain $\Omega$.

Then, by definition, an extension $(\tilde{v}_{\varepsilon}$, $\tilde{P}_{\varepsilon})\in W_0^{1,p}(\Omega)^3\times L^{p^\prime}_0(\Omega)$ of $(\tilde{u}_{\varepsilon}$, $\tilde{p}_{\varepsilon})$ is defined on $\Omega$ and coincides with $(\tilde{u}_{\varepsilon}$, $\tilde{p}_{\varepsilon})$ on $\widetilde{\Omega}_{\varepsilon}$.

In order to simplify the notation, we define $S$ as the $p$-Laplace operator $$S(\xi)=\left\vert \xi \right\vert^{p-2}\xi,\text{\ \ \ }\forall \xi \in \mathbb{R}^{3\times 3}_{{\rm sym}}, \quad 1< p<+\infty.$$

Our main result referred to the asymptotic behavior of a solution of (\ref{2})-(\ref{no-slip_tilde}) is given by the following theorem. 
\begin{theorem}\label{MainTheorem} 
Assume $9/5\leq p< +\infty$. We distingue three cases depending on the relation between the parameter $\eta_\varepsilon$ with respect to $\varepsilon$:
\begin{itemize}
\item[i)] If $\eta_{\varepsilon}\approx \varepsilon$, with $\eta_\varepsilon/\varepsilon\to \lambda$, $0<\lambda<+\infty$, then the extension $(\eta_\varepsilon^{-{p\over p-1}}\tilde{v}_{\varepsilon},\tilde{P}_{\varepsilon})$ of a solution of (\ref{2})-(\ref{no-slip_tilde}) converges  weakly to  $(\tilde{v},\tilde{P})$ in $W^{1,p}(0,h_{\rm max};L^p(\omega)^3) \times L^{p^\prime}_0(\omega)$ with $\tilde v_3=0$. Moreover, it holds that $\tilde P\in W^{1,p'}(\omega)$ and $(\tilde{V}^\prime,\tilde{P})$ is the unique solution of the nonlinear Reynold problem \begin{equation}
\left\{
\begin{array}
[c]{r@{\;}c@{\;}ll}%
\displaystyle\tilde V^\prime(x^\prime) &
= &
\displaystyle{1\over \mu}A^{\lambda}\left(f^\prime(x^{\prime})-\nabla_{x^\prime} \tilde{P}(x^{\prime}) \right) \text{\ in \ }\omega,\\
\displaystyle {\rm div}_{x^{\prime}}\tilde V^\prime(x^\prime) & = & 0 \text{\ in \ }\omega,\\
\displaystyle \tilde V^\prime(x^\prime) \cdot n & = & 0 \text{\ in \ }\partial\omega,
\end{array}
\right. \label{Main1}%
\end{equation}
where $\tilde V^\prime(x^\prime)=\int_0^{h_{\rm max}}\tilde{v}^\prime(x^\prime,y_3)\,dy_3$ and $A^{\lambda}:\mathbb{R}^2\to \mathbb{R}^2$ is monotone, coercive and defined by 
\begin{equation}\label{def_A_lambda}
A^{\lambda}(\xi^\prime)=\int_{Y}w^{\xi^\prime}(y)\,dy,\quad \forall\,\xi^\prime\in\mathbb{R}^2,
\end{equation}
where   $w^{\xi^\prime}(y)$, for every $\xi^\prime\in\mathbb{R}^2$, denote the unique solution in $W^{1,p}_\sharp(Y)^3$ of the local Stokes problem in 3D
\begin{equation}\label{local_3D}
\left\{\begin{array}{rcl}\displaystyle
-{\rm div}_{\lambda}S\left(\mathbb{D}_\lambda[w^{\xi^\prime}]\right)+\nabla_{\lambda}\pi^{\xi^\prime}&=&\xi^\prime\quad\text{ in \ } Y\,,\\
\displaystyle
{\rm div}_{\lambda} w^{\xi^\prime}&=&0\quad\text{ in \ } Y\,,\\
w^{\xi^\prime}&=&0\quad\text{ in \ } y_3=0,h(y')\,,\\
w^{\xi^\prime}, \pi^{\xi^\prime}\ Y^\prime-\text{periodic}.
\end{array}\right.
\end{equation}
where $\mathbb{D}_{\lambda}\left[\cdot\right]=\lambda \mathbb{D}_{y^\prime}\left[\cdot\right]+ \partial_{y_3}\left[\cdot\right]$,  $\nabla_{\lambda}=(\lambda\nabla_{y^\prime},  \partial_{y_3})^t$ 
 and ${\rm div}_\lambda=\lambda{\rm div}_{y^\prime} + \partial_{y_3}$.
\item[ii)] if $\eta_{\varepsilon}\ll \varepsilon$, then the extension $(\eta_\varepsilon^{-{p\over p-1}}\tilde{v}_{\varepsilon},\tilde{P}_{\varepsilon})$ of a solution of (\ref{2})-(\ref{no-slip_tilde}) converges  weakly to  $(\tilde{v},\tilde{P})$ in $W^{1,p}(0,h_{\rm max};L^p(\omega)^3) \times L^{p^\prime}_0(\omega)$ with $\tilde v_3=0$. Moreover, it holds that $\tilde P\in W^{1,p'}(\omega)$ and $(\tilde{V}^\prime,\tilde{P})$ is the unique solution of the nonlinear Reynolds problem
\begin{equation}
\left\{
\begin{array}
[c]{r@{\;}c@{\;}ll}%
\displaystyle \tilde{V}^\prime(x^{\prime}) &
= &
\displaystyle {1\over 2^{p'\over 2}(p'+1)\mu}A^0\left(\tilde f'(x')-\nabla_{x'}\tilde p(x')\right) \text{\ in \ }\omega,\\
\displaystyle {\rm div}_{x^{\prime}}\, \tilde{V}^\prime(x^{\prime}) & = & 0 \text{\ in \ }\omega,\\
\displaystyle \tilde{V}^\prime(x^{\prime}) \cdot n & = & 0 \text{\ in \ }\partial\omega,
\end{array}
\right. \label{Main2}%
\end{equation}
where $\tilde V(x^\prime)=\int_0^{h_{\rm max}}\tilde{v}(x^\prime,y_3)\,dy_3$ and $A^{0}:\mathbb{R}^2\to \mathbb{R}^2$ is monotone, coercive and   defined by  
\begin{equation}\label{def_A_0}
A^{0}(\xi^\prime)=\int_{Y'}h(y')^{p'+1}\left\vert \xi^\prime+\nabla_{y'}\pi^{\xi^\prime}\right\vert^{p'-2}\left(\xi^\prime+\nabla_{y'}\pi^{\xi^\prime} \right)dy',\quad \forall\,\xi^\prime\in\mathbb{R}^2,
\end{equation}
where,  $\pi^{\xi^\prime}(y')$, for every $\xi^\prime\in\mathbb{R}^2$, denote the unique solution in $W^{1,p'}_\sharp(Y^\prime)\cap L^{p'}_0(Y')$ of the local Reynolds problem in 2D
\begin{equation}\label{local_2D_infty}
\left\{\begin{array}{rcl}\displaystyle
{\rm div}_{y^\prime}\left(h(y')^{p'+1}\left\vert \xi^\prime+\nabla_{y'}\pi^{\xi^\prime}\right\vert^{p'-2}\left(\xi^\prime+\nabla_{y'}\pi^{\xi^\prime} \right)\right)&=&0\quad\text{ in \ } Y^\prime,\\
\left(h(y')^{p'+1}\left\vert \xi^\prime+\nabla_{y'}\pi^{\xi^\prime}\right\vert^{p'-2}\left(\xi^\prime+\nabla_{y'}\pi^{\xi^\prime} \right)\right)\cdot  n&=&0\quad\text{ in \ } \partial Y^\prime.
\end{array}\right. 
\end{equation}
\item[iii)] If $\eta_{\varepsilon}\gg \varepsilon$, then the extension $(\eta_\varepsilon^{-{p\over p-1}}\tilde{v}_\varepsilon,\tilde{P}_\varepsilon)$ of the solution of (\ref{2})-(\ref{no-slip_tilde}) converges  weakly to $(\tilde v, \tilde P)$ in $W^{1,p}(0,h_{\rm min};L^p(\omega)^3)\times L^{p^\prime}_0(\omega)$, with $\tilde v_3=0$. Moreover, it holds that $\tilde P\in W^{1,p'}(\omega)$ and $(\tilde{V}^\prime,\tilde{P})$ is the unique solution of the nonlinear Reynolds problem
\begin{equation}\label{thm_Homogenized_3}
\left\{\begin{array}{l}\displaystyle
 \tilde{V}^\prime(x^{\prime}) 
=
 {h_{\rm min}^{p'+1}\over 2^{p'\over 2}(p'+1)\mu^{p'-1}}\left|\tilde f'(x')-\nabla_{x'}\tilde p(x')\right|^{p'-2}\!\!\left(\tilde f'(x')-\nabla_{x'}\tilde p(x')\right)\,,
\\
\displaystyle {\rm div}_{x^{\prime}}\, \tilde{V}^\prime(x^{\prime}) = 0 \text{\ in \ }\omega,\\
\displaystyle \tilde{V}^\prime(x^{\prime}) \cdot n = 0 \text{\ in \ }\partial\omega\,,\end{array}\right.
\end{equation}
where $\tilde V(x^\prime)=\int_0^{h_{\rm min}}\tilde{v}(x^\prime,y_3)\,dy_3$.
\end{itemize}
\end{theorem}

\begin{remark} The monotonicity and coerciveness properties of $A^\lambda$ and $A^0$ given by (\ref{def_A_lambda}) and (\ref{def_A_0}), respectively, can be found in Bourgeat {\it et al.} \cite{Marusic1}.
\end{remark}
\begin{remark} This is a preliminary step towards a complete generalization of the papers of Bayada and Chambat \cite{Bayada_Chambat_2,Bayada_Chambat} in order to consider rough surfaces of type $\eta_\varepsilon h(x',x'/\varepsilon)$ (locally periodic oscillatory boundaries), which are more practical from the engineering point of view. We think that this could be successfully managed by an adaptation of the recent version of the unfolding method introduced by Arrieta and Villanueva-Pesqueira \cite{Arrieta}, which will be object of a future study.

 \end{remark}
\section{{\it A priori} estimates}\label{S4}
Let us begin with the classical Poincar\'e and  Korn inequalities.

\begin{lemma} (Poincar\'e's inequality) For $w\in W^{1,p}_0(\Omega_\varepsilon)^3$, $1\leq p<+\infty$, 
\begin{equation}\label{Poincare}
\|w\|_{L^p(\Omega_\varepsilon)^3}\leq C\eta_\varepsilon\|\partial_{x_3}w\|_{L^p(\Omega_\varepsilon)^3}\,,
\end{equation}
where $C$ is independent of $w$ and $\varepsilon$.
\end{lemma}

\begin{lemma} (Korn's inequality) For $w\in W^{1,p}_0(\Omega_\varepsilon)^3$, $1<p<+\infty$,
\begin{equation}\label{Korn}
\|D w\|_{L^p(\Omega_\varepsilon)^{3\times 3}}\leq C\|\mathbb{D}[w]\|_{L^p(\Omega_\varepsilon)^{3\times 3}}\,,
\end{equation}
where $C$ is independent of $w$ and $\varepsilon$.
\end{lemma}

\begin{proof} See Lemmas 1.2 and 1.3 in Mikeli\'c  and Tapiero \cite{MT}.
\end{proof}

Let us obtain some {\it a priori} estimates for velocities $u_\varepsilon$ and $\tilde{u}_{\varepsilon}$.
\begin{lemma}\label{Lemma_a3} Assume that $9/5\leq p<+\infty$. 
There exists a constant $C$ independent of $\varepsilon$, such that a solution $u_\varepsilon$ of problem (\ref{1})-(\ref{CF}) and the corresponding rescaled solution, $\tilde{u}_{\varepsilon}$, of the problem (\ref{2})-(\ref{no-slip_tilde}) satisfy
\begin{equation}\label{estim_u_1}
\left\Vert {u}_{\varepsilon}\right\Vert_{L^p({\Omega}_{\varepsilon})^3}\leq C\eta_\varepsilon^{{2p-1\over p(p-1)}+1},\quad \left\Vert \mathbb{D}\left[{u}_{\varepsilon}\right]\right\Vert_{L^p({\Omega}_{\varepsilon})^{{3\times3}}}\leq C\eta_\varepsilon^{2p-1\over p(p-1)}\,,
\end{equation}
\begin{equation}\label{estim_Du_1}
\left\Vert D{u}_{\varepsilon}\right\Vert_{L^p({\Omega}_{\varepsilon})^{{3\times3}}}\leq C\eta_\varepsilon^{2p-1\over p(p-1)}\,,
\end{equation}
\begin{equation}\label{a3}
\left\Vert \tilde{u}_{\varepsilon}\right\Vert_{L^p(\widetilde{\Omega}_{\varepsilon})^3}\leq C\eta_\varepsilon^{p\over p-1},\quad \left\Vert \mathbb{D}_{\eta_\varepsilon}\left[\tilde{u}_{\varepsilon}\right]\right\Vert_{L^p(\widetilde{\Omega}_{\varepsilon})^{{3\times3}}}\leq C\eta_\varepsilon^{1\over p-1},
\end{equation}
\begin{equation}\label{Derivada_U}
\left\Vert D_{\eta_\varepsilon}\tilde{u}_{\varepsilon}\right\Vert_{L^p(\widetilde{\Omega}_{\varepsilon})^{{3\times3}}}\leq C\eta_\varepsilon^{1\over p-1}.
\end{equation}
\end{lemma}
\begin{proof}
Multiplying by ${u}_{\varepsilon}$ in the first equation of (\ref{1}) and integrating over ${\Omega}_{\varepsilon}$, we have
\begin{eqnarray}\label{a1}
\mu \|\mathbb{D}\left[{u}_{\varepsilon}\right]\|^p_{L^{p}(\Omega_\varepsilon)^{3\times 3}}=\int_{{\Omega}_{\varepsilon}}f\cdot {u}_{\varepsilon}\,dx.
\end{eqnarray}
Using H${\rm \ddot{o}}$lder's inequality and the assumption of $f$, we obtain that 
\begin{eqnarray*}
\int_{{\Omega}_{\varepsilon}}f\cdot {u}_{\varepsilon}\,dx\leq C\eta_\varepsilon^{1\over p'} \left\Vert {u}_{\varepsilon} \right\Vert_{L^p({\Omega}_{\varepsilon})^3},
\end{eqnarray*}
and by (\ref{a1}), we have
\begin{equation*}\label{a2}
\left\Vert \mathbb{D}\left[{u}_{\varepsilon}\right]\right\Vert_{L^p({\Omega}_{\varepsilon})^{{3\times3}}}^p\leq C\eta_\varepsilon^{1\over p'} \left\Vert {u}_{\varepsilon} \right\Vert_{L^p({\Omega}_{\varepsilon})^3}.
\end{equation*}

Taking into account  (\ref{Poincare}) and (\ref{Korn}), we obtain the second estimate in (\ref{estim_u_1}). 

Consequently, from (\ref{Korn}) and the second estimate in (\ref{estim_u_1}), we get (\ref{estim_Du_1}). Finally, taking into account   (\ref{Poincare}) and (\ref{estim_Du_1}), we obtain the first estimate in (\ref{estim_u_1}).

By means of the dilatation (\ref{dilatacion}), we get (\ref{a3}) and (\ref{Derivada_U}).\par
\end{proof}

\begin{lemma}\label{lem_estim_nabla_pressure} Assume that $9/5\leq p<+\infty$. There exists a constant $C$ independen of $\varepsilon$, such that a solution $\tilde{p}_{\varepsilon}$ of the problem (\ref{2})-(\ref{no-slip_tilde}) satisfies
\begin{equation}\label{estim_nabla_p_1}\|\nabla_{\eta_\varepsilon}\tilde p_\varepsilon\|_{W^{-1,p'}(\widetilde \Omega_\varepsilon)^3}\leq C\,.
\end{equation}
\end{lemma}
\begin{proof} From system (\ref{2}), we have that (brackets are for the duality products between $W^{-1,p^\prime}$ and $W_0^{1,p}$)
\begin{equation}\label{equality_duality_1}
\begin{array}{rl}
\displaystyle
\langle \nabla_{\eta_\varepsilon}\tilde p_\varepsilon, \tilde\varphi\rangle_{\widetilde \Omega_\varepsilon}   =&\displaystyle
-\mu\int_{\widetilde\Omega_\varepsilon}  S( \mathbb{D}_{\eta_\varepsilon}\left[\tilde {u}_\varepsilon\right]): D_{\eta_\varepsilon}\tilde \varphi\,dx'dy_3\\
&\displaystyle +\int_{\widetilde\Omega_\varepsilon} f\cdot \tilde \varphi\,dx'dy_3 -\int_{\widetilde\Omega_\varepsilon}  (\tilde u_\varepsilon\cdot \nabla_{\eta_\varepsilon}) \tilde u_\varepsilon\, \tilde\varphi\,dx'dy_3\, .
\end{array}
\end{equation}
for every $\tilde\varphi\in W^{1,p}_0(\widetilde\Omega_\varepsilon)^3$. By the second estimate in (\ref{a3}), we have  
\begin{equation}\label{inertial_123_1}\begin{array}{rl}\displaystyle
\left| \mu\int_{\widetilde \Omega_\varepsilon} S( \mathbb{D}_{\eta_\varepsilon}\left[\tilde {u}_\varepsilon\right]) : D_{\eta_\varepsilon} \tilde\varphi\,dx'dy_3\right|\leq &\!\!\!\!\displaystyle
 \|\mathbb{D}_{\eta_\varepsilon}[\tilde u_\varepsilon]\|^{p-1}_{L^{p}(\widetilde \Omega_\varepsilon)^{3\times 3}}\|D_{\eta_\varepsilon} \tilde\varphi\|_{L^p(\widetilde \Omega_\varepsilon)^{3\times 3}}\\
\\
\leq &\!\!\!\!\displaystyle
{1\over \eta_\varepsilon}\|\mathbb{D}_{\eta_\varepsilon}[\tilde u_\varepsilon]\|^{p-1}_{L^{p}(\widetilde \Omega_\varepsilon)^{3\times 3}}\| \tilde\varphi\|_{W^{1,p}_0(\widetilde \Omega_\varepsilon)^{3\times 3}}\\
\\
\leq & \!\!\!\!C\|\tilde\varphi\|_{W^{1,p}_0(\widetilde \Omega_\varepsilon)^3}\,,\\
\\
\displaystyle
\left|\int_{\widetilde \Omega_\varepsilon}f\cdot \tilde\varphi\,dx'dy_3\right|\leq &\displaystyle \!\!\!\! C\|\tilde\varphi\|_{W^{1,p}_0(\widetilde \Omega_\varepsilon)^3}\,.
\end{array}
\end{equation}
Hence, to derive estimates for $\nabla_{\eta_\varepsilon} \tilde p_\varepsilon$ from (\ref{equality_duality_1}), we just need to consider the initial terms, which can be written
\begin{equation}\label{inertial_1_1}\begin{array}{l}
\displaystyle \int_{\widetilde \Omega_\varepsilon} (\tilde u_\varepsilon\cdot \nabla_{\eta_\varepsilon}) \tilde u_\varepsilon\,\tilde \varphi\, dx'dy_3= \displaystyle -\int_{\widetilde \Omega_\varepsilon} \tilde u_\varepsilon\tilde \otimes \tilde u_\varepsilon:D_{x'} \tilde \varphi\,dx'dy_3\\
\displaystyle \qquad+{1\over \eta_\varepsilon}\left(\int_{\widetilde \Omega_\varepsilon}\partial_{y_3} \tilde u_{\varepsilon,3}\tilde  u_\varepsilon\cdot  \tilde \varphi \,dx'dy_3+ \int_{\widetilde \Omega_\varepsilon}\tilde u_{\varepsilon,3}\partial_{y_3} \tilde u_\varepsilon\cdot \tilde\varphi\,dx'dy_3\right)\,,
\end{array}
\end{equation}
where $(u \tilde \otimes w)_{ij}=u_i w_j$, $i=1,2$, $j=1,2,3$.

We consider separately the two terms in the right-hand side of (\ref{inertial_1_1}):

(i) Estimate of the first part of the right-hand side of (\ref{inertial_1_1}) has the form
$$\|\tilde u_\varepsilon\|^2_{L^{q'}(\widetilde \Omega_\varepsilon)^3}\|D_{x'}\tilde \varphi\|_{L^{p}(\widetilde \Omega_\varepsilon)^{3\times 2}}\,,$$
with $2/q'+1/p=1$.

We introduce the interpolation parameter $\theta={p^*(p-1)-2p\over 2(p^*-p)}$ where 
$p^*={3p\over (3-p)}$ if $9/5\leq p<3$, $p^*\in [p,+\infty)$ if $p=3$ and $p^*\in[p,+\infty]$ if $p>3$.

For $9/5\leq p<3$, we have that $0\leq \theta\leq 1$ such that
$${1\over q'}={\theta\over p}+{1-\theta\over p^*}\,.$$
We have the interpolation 
$$\|\tilde u_\varepsilon\|_{L^{q'}(\widetilde \Omega_\varepsilon)^3}\leq \|\tilde u_\varepsilon\|^{\theta}_{L^p(\widetilde \Omega_\varepsilon)^3} \|\tilde u_\varepsilon\|^{1-\theta}_{L^{p^*}(\widetilde \Omega_\varepsilon)^3}\,,$$
and by the the Sobolev embedding, $W^{1,p}_0 \hookrightarrow L^{p^*}$, and the first estimate in (\ref{a3}) and estimate (\ref{Derivada_U}), we obtain
$$\begin{array}{rl}
\displaystyle \|\tilde u_\varepsilon\|_{L^{q'}(\widetilde \Omega_\varepsilon)^3}\leq &\displaystyle
\|\tilde u_\varepsilon\|^{\theta}_{L^p(\widetilde \Omega_\varepsilon)^3} \|D\tilde u_\varepsilon\|^{1-\theta}
_{L^{p}(\widetilde \Omega_\varepsilon)^{3\times 3}}\leq \|\tilde u_\varepsilon\|_{W_0^{1,p}(\widetilde \Omega_\varepsilon)^3} \leq C \eta_\varepsilon^{1 \over p-1}\,,
\end{array}$$
and then, 
\begin{equation*}\label{inertial_3_1}
\left|\int_{\widetilde \Omega_\varepsilon}\tilde  u_\varepsilon\tilde \otimes\tilde  u_\varepsilon:D_{x'} \tilde \varphi\,dx'dy_3\right|\leq C\eta_\varepsilon^{2\over p-1} \|\tilde\varphi\|_{W^{1,p}_0(\widetilde \Omega_\varepsilon)^3}\,.
\end{equation*}

For $p\ge 3$, we take $p^*=p$ and we have
\begin{equation*}\label{inertial_3_1}
\left|\int_{\widetilde \Omega_\varepsilon}\tilde  u_\varepsilon\tilde \otimes\tilde  u_\varepsilon:D_{x'} \tilde \varphi\,dx'dy_3\right|\leq C\eta_\varepsilon^{2p\over p-1} \|\tilde\varphi\|_{W^{1,p}_0(\widetilde \Omega_\varepsilon)^3}\,.
\end{equation*}

(ii) Estimate of the second part of the right-hand side of (\ref{inertial_1_1}) has the form
$${C\over \eta_\varepsilon}\|\partial_{y_3} \tilde u_{\varepsilon}\|_{L^{p}(\widetilde \Omega_\varepsilon)^3}\|\tilde u_\varepsilon\|_{L^{q'}(\widetilde \Omega_\varepsilon)^3}\|\tilde \varphi\|_{L^{q'}(\widetilde \Omega_\varepsilon)^3}\,,$$
with $2/q'+1/p=1$.

For $9/5\leq p<3$, working as in item (i), we have 
$$\|\tilde u_\varepsilon\|_{L^{q'}(\widetilde\Omega_\varepsilon)^3}\leq C\eta_{\varepsilon}
^{1\over p-1}\,,\quad \|\tilde \varphi\|_{L^{q'}(\widetilde\Omega_\varepsilon)^3}\leq \|\tilde \varphi\|_{W^{1,p}_0
(\widetilde\Omega_\varepsilon)^3}\,,$$
and by estimate (\ref{Derivada_U}), we get
\begin{equation*}\label{inertial_4_1}{1\over \eta_\varepsilon}\left|\int_{\widetilde \Omega_\varepsilon} \left(\partial_{y_3} \tilde u_{\varepsilon,3} \tilde u_\varepsilon \tilde \varphi\,dx'dy_3 +\int_{\widetilde \Omega_\varepsilon}\tilde  u_{\varepsilon,3}\partial_{y_3} \tilde u_\varepsilon\, \tilde \varphi\,dx'dy_3\right) \right|\leq 
C\eta_\varepsilon^{2\over p-1}\|\tilde\varphi\|_{W^{1,p}_0(\widetilde \Omega_\varepsilon)^3}\,.\end{equation*}
For $p\ge 3$, we take $p^*=p$ and we have
\begin{equation*}\label{inertial_4_1}{1\over \eta_\varepsilon}\left|\int_{\widetilde \Omega_\varepsilon} \left(\partial_{y_3} \tilde u_{\varepsilon,3} \tilde u_\varepsilon \tilde \varphi\,dx'dy_3 +\int_{\widetilde \Omega_\varepsilon}\tilde  u_{\varepsilon,3}\partial_{y_3} \tilde u_\varepsilon\, \tilde \varphi\,dx'dy_3\right) \right|\leq 
C\eta_\varepsilon^{2p\over p-1}\|\tilde\varphi\|_{W^{1,p}_0(\widetilde \Omega_\varepsilon)^3}\,.\end{equation*}
Then, from (\ref{inertial_1_1}) we can deduce that for $9/5\leq p<3$, we obtain
$$\left|\int_{\widetilde \Omega_\varepsilon} (\tilde u_\varepsilon\cdot \nabla_{\eta_\varepsilon})\tilde u_\varepsilon\,\tilde \varphi\, dx'dy_3\right|\leq  C\eta_\varepsilon^{2\over p-1}\|\tilde \varphi\|_{W^{1,p}_0
(\widetilde\Omega_\varepsilon)^3}\,,$$
and for $p\ge 3$, we get
$$\left|\int_{\widetilde \Omega_\varepsilon} (\tilde u_\varepsilon\cdot \nabla_{\eta_\varepsilon})\tilde u_\varepsilon\,\tilde \varphi\, dx'dy_3\right|\leq  C\eta_\varepsilon^{2p\over p-1}\|\tilde \varphi\|_{W^{1,p}_0
(\widetilde\Omega_\varepsilon)^3}\,.$$
Taking into account the previous estimates with $\eta_\varepsilon\ll 1$ and (\ref{inertial_123_1}) in (\ref{equality_duality_1}), for $9/5\leq p<+\infty$, we have
\begin{equation*}\label{estimFQ_1}
|\langle \nabla_{\eta_\varepsilon} \tilde p_\varepsilon,\tilde\varphi \rangle_{\widetilde\Omega_\varepsilon}|\leq C\|\tilde\varphi \|_{W^{1,p}_0(\widetilde \Omega_\varepsilon)^3}\,,
\end{equation*}
and so we have the  estimate (\ref{estim_nabla_p_1}).\par  \end{proof}
In order to estimate the pressure, since $\widetilde \Omega_\varepsilon$ is a bounded Lipschitz domain, we have
$$\|\tilde p_\varepsilon\|_{L^{p'}_0(\widetilde \Omega_\varepsilon)}\leq C(\widetilde \Omega_\varepsilon) \|\nabla \tilde p_\varepsilon\|_{W^{-1,p'}(\widetilde\Omega_\varepsilon)^3}.$$
We take into account that the constant depends on the domain, i.e. it depends on $\varepsilon$. Thus, we can not obtain an estimate of the pressure in a fixed domain in order to prove convergence. So we have to define a continuation of the pressure to $\Omega$ in order to prove convergence.

\subsection{The Extension of ($\tilde{u}_{\varepsilon}$, $\tilde{p}_{\varepsilon}$) to the whole domain $\Omega$}
In this section, we will extend the solution $(\tilde{u}_{\varepsilon},\tilde{p}_{\varepsilon})$ to the whole domain $\Omega$. It is easy to extend the velocity by zero in $\Omega\backslash \widetilde{\Omega}_{\varepsilon}$ (this is compatible with the no-slip boundary condition on $\partial \widetilde{\Omega}_{\varepsilon}$). We will denote by $\tilde{v}_{\varepsilon}$ the continuation  of $\tilde u_\varepsilon$ in $\Omega$. It is well known that extension by zero preserves $L^p$ and $W_0^{1,p}$ norms for $1<p<+\infty$. We note that the extension $\tilde{v}_{\varepsilon}$ belongs to $W_0^{1,p}(\Omega)^3$.

Extending the pressure is a much more difficult task. Tartar \cite{Tartar} introduced a continuation of the pressure for a flow in a porous media. This construction applies to periodic holes in a domain $\widetilde \Omega_\varepsilon$ when each hole is strictly contained into the periodic cell. In this context, we can not use directly this result because the ``holes" are along the boundary $\widetilde \Sigma_\varepsilon$ of $\widetilde \Omega_\varepsilon$, and moreover the scale of the vertical direction is smaller than the scales of the horizontal directions. This fact will induce several limitations in the results obtained by using the method, especially in view of the convergence for the pressure. In this sense, for the case of Newtonian fluids, Bayada and Chambat  \cite{Bayada_Chambat}  and  Mikeli\'c \cite{Mikelic2} introduced an operator $R^\varepsilon$  generalizing the results of Tartar \cite{Tartar} to this context.  In our case, we need an operator $R_p^\varepsilon$ between $W^{1,p}(\Omega)^3$ and $W^{1,p}(\widetilde\Omega_\varepsilon)^3$ with similar properties. 

Let us introduce some notation. We consider that the domain $\omega$ is covered by a rectangular mesh of size $\varepsilon$: for $k'\in \mathbb{Z}^2$, each cell $Y'_{k',\varepsilon}=\varepsilon k'+\varepsilon Y'$.  We define the thin domain 
$$Q_{\varepsilon}=\{x\in\mathbb{R}^3\,:\, x'\in\omega,\ 0<x_3< \eta_\varepsilon h_{\rm max}\}\,,$$
and the corresponding  cubes of size $\varepsilon$ and height $\eta_\varepsilon h_{\rm max}$ given by $Q_{k',\varepsilon}=Y'_{k',\varepsilon}\times (0,\eta_\varepsilon h_{\rm max})$.  We also define $\widetilde Q_{k',\varepsilon}=Y'_{k',\varepsilon}\times (0, h_{\rm max})$.

According to the definition of the basic cell $Y$ defined in (\ref{cell}), we also define $Y_{k',\varepsilon}=Y'_{k',\varepsilon}\times (0,h(y'))$ for $k'\in\mathbb{Z}^2$. We also consider  a smooth surface included in $Y$ and surrounding the hump such that $Y$ is split into two areas $Y_f$ and $Y_m$ (see Figure \ref{fig:periodic_cell_2D_3D}).

\begin{figure}[h!]
\begin{center}
\includegraphics[width=9.5cm]{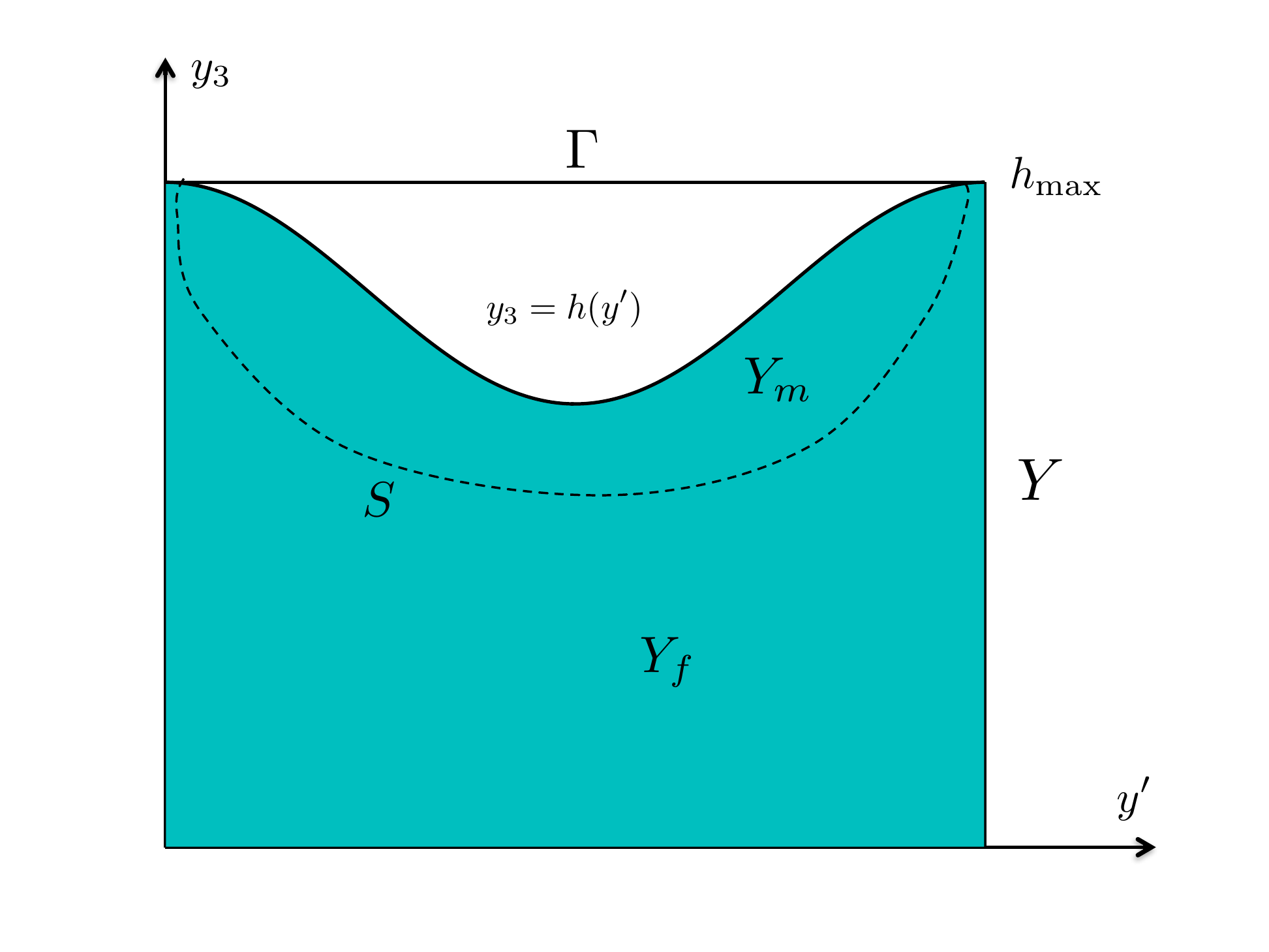}
 \vspace{-1cm}
 \end{center}
\caption{Basic cell Y}
\label{fig:periodic_cell_2D_3D}
\end{figure}

We also define  the following sets:
$$\begin{array}{l}
\Pi=Y'\times (0,h_{\rm max}),\\
\Pi^-=Y'\times (0,h_{\rm min}),\\
Y_s=\Pi\setminus (Y_m\cup Y_f),\\
S=\partial Y_m\cap \partial Y_f.
\end{array}
$$

We suppose from now on the following assumptions:
\begin{itemize}
\item[(H1)] the surface roughness is made of detached smooth humps periodically given on the upper part of the gap,
\item[(H2)] $\omega$ is covered by an exact finite number of periodic  sets $Y_{k',\varepsilon}$. Thus, we define $T_\varepsilon=\{k'\in\mathbb{Z}^2$\,:\, $\omega\cap Y_{k',\varepsilon}\neq \emptyset\}$, 
\item[(H3)] $\partial Y_m$ is a $C^1$ manifold.
\end{itemize}

Generalizing Bayada and Chambat \cite{Bayada_Chambat}, we get the following.
\begin{lemma}\label{lemma_new1}
For given $\tilde \varphi\in W^{1,p}(\Pi)^3$, $1<p<+\infty$, such that $\tilde \varphi=0$ on $\Gamma$, there exists $\tilde w$ in $W^{1,p}(Y_m)^3$ such that:
$$\tilde w_{|_S}=\tilde \varphi_{|_S}\quad\hbox{and}\quad \tilde w_{|_{\partial Y_m\setminus S}}=0\,.$$
Moreover, there exists a constant $C$ which does not depend on $\tilde\varphi$ such that:
\begin{equation}\label{lema_1_bayada_estim}\left\{\begin{array}{l}\displaystyle
\|\tilde w\|_{W^{1,p}(Y_m)^3}\leq C\|\tilde \varphi\|_{W^{1,p}(\Pi)^3},\\
\\
{\rm div}_{\eta_\varepsilon}\, \tilde \varphi=0\Rightarrow {\rm div}_{\eta_\varepsilon}\,\tilde w=0\,.\end{array}\right.
\end{equation}
\end{lemma}

\begin{proof}
The proof  is very similar to that given in \cite{Bayada_Chambat} for the case $p=2$. In addition to the technique used in \cite{Bayada_Chambat}, one needs $L^p$-regularity for the Stokes equation.\par
\end{proof}

\begin{lemma}\label{Operator} There exists an operator $R_p^\varepsilon: W^{1,p}_0(Q_\varepsilon)^3\to W^{1,p}_0(\Omega_\varepsilon)^3$, $1<p<+\infty$, such that:
\begin{itemize}
\item[1. ] $\varphi\in W^{1,p}_0(\Omega_\varepsilon)^3\Rightarrow R_p^{\varepsilon}(\varphi)=\varphi$\,,
\item[2. ] ${\rm div}\,\varphi=0 \Rightarrow  {\rm div}\,R_p^{\varepsilon}(v)=0$\,, 
\item[3. ] For any $\varphi\in W^{1,p}_0(Q_\varepsilon)^3$ (the constant $C$ is independent of $\varphi$ and $\varepsilon$), we have
$$\begin{array}{l}\displaystyle
\|R_p^\varepsilon(\varphi)\|_{L^p(\Omega_\varepsilon)^3}\leq \displaystyle C\Big(\|\varphi\|_{L^p(Q_\varepsilon)^3}\\
\displaystyle
\qquad\qquad\qquad\qquad\quad+\,\varepsilon \|D_{x'}\varphi\|_{L^p(Q_\varepsilon)^{3\times 2}}+\,\eta_\varepsilon \|\partial_{x_3}\varphi\|_{L^p(Q_\varepsilon)^{3}}\Big)\,,\\
\\
\displaystyle \|D_{x'} R_p^\varepsilon(\varphi)\|_{L^p(\Omega_\varepsilon)^{3\times 2}}\leq \displaystyle C\Big({1 \over \varepsilon}\|\varphi\|_{L^p(Q_\varepsilon)^3}\\
\qquad\qquad\qquad\qquad\qquad\quad +\, \|D_{x'}\varphi\|_{L^p(Q_\varepsilon)^{3\times 2}}+\,{\eta_\varepsilon \over \varepsilon} \|\partial_{x_3}\varphi\|_{L^p(Q_\varepsilon)^{3}}\Big)\,,\\
\\
\displaystyle \|\partial_{x_3} R_p^\varepsilon(\varphi)\|_{L^p(\Omega_\varepsilon)^{3}}\leq \displaystyle C\Big({1 \over \eta_\varepsilon}\|\varphi\|_{L^p(Q_\varepsilon)^3}\\
\displaystyle\qquad\qquad\qquad\qquad\qquad\quad+\, {\varepsilon \over \eta_\varepsilon}\|D_{x'}\varphi\|_{L^p(Q_\varepsilon)^{3\times 2}}+ \|\partial_{x_3}\varphi\|_{L^p(Q_\varepsilon)^{3}}\Big)\,.
\end{array}$$
\end{itemize}
\end{lemma}
\begin{proof} For any $\tilde \varphi\in W^{1,p}_0(\Pi)^3$ such that $\tilde \varphi=0$ on $\Gamma$, Lemma \ref{lemma_new1} allows us to define $R_p(\tilde \varphi)\in W^{1,p}(\Pi)^3$ by
$$R_p(\tilde \varphi)=\left\{
\begin{array}{rcl}
\tilde \varphi & \hbox{if} & y\in Y_f\,,\\
\tilde w & \hbox{if} &y\in Y_m\,,\\
0 & \hbox{if} & y\in Y_s\,,\\
\end{array}\right.
$$
which satisfies 
\begin{equation}\label{estim_lem_2}
\int_{\Pi}|R_p(\tilde\varphi)|^p\,dy+ \int_{\Pi}|D_y R_p(\tilde\varphi)|^p\,dy\leq C\left(\int_{\Pi}|\tilde\varphi|^p\,dy+ \int_{\Pi}|D_y \tilde\varphi|^p\,dy\right)\,.
\end{equation}

For every $k'\in T_\varepsilon$, by the change of variables
\begin{equation*}\label{change_res}
k'+y'={x'\over \varepsilon},\ y_3={x_3\over \eta_\varepsilon},\ dy={dx\over \varepsilon^2\eta_\varepsilon},\ \partial_{y'}=\varepsilon\,\partial_{x'},\ \partial_{y_3}=\eta_\varepsilon\,\partial_{x_3}\,,
\end{equation*}
we rescale (\ref{estim_lem_2}) from $\Pi$ to $Q_{k',\varepsilon}$. This yields that, for every function $\varphi\in W^{1,p}(Q_{k',\varepsilon})^3$, one has
$$\begin{array}{l}\displaystyle
\int_{Q_{k',\varepsilon}}|R_p(\varphi)|^p\,dx+ \varepsilon^p \int_{Q_{k',\varepsilon}}|D_{x'} R_p(\varphi)|^p\,dx+ \eta_\varepsilon^p\int_{Q_{k',\varepsilon}}|\partial_{x_3} R_p(\varphi)|^p\,dx\\
\\
\displaystyle
 \leq C\left(\int_{Q_{k',\varepsilon}}|\varphi|^p\,dx+ \varepsilon^p\int_{Q_{k',\varepsilon}}|D_{x'} \varphi|^p\,dx+\eta_\varepsilon^p\int_{Q_{k',\varepsilon}}|\partial_{x_3} \varphi|^p\,dx.\right)\\
 \end{array}
 $$
 We define $R_p^\varepsilon$ by applying $R_p$ to each period $Q_{k',\varepsilon}$. Summing the previous inequalities for all the periods $Q_{k',\varepsilon}$, and taking into account that from (H2) we have $Q_{\varepsilon}=\cup_{k'\in T_\varepsilon} Q_{k',\varepsilon}$, gives
 $$\begin{array}{l}\displaystyle
\int_{Q_\varepsilon}|R_p^\varepsilon(\varphi)|^p\,dx+ \varepsilon^p \int_{Q_\varepsilon}|D_{x'} R_p^\varepsilon(\varphi)|^p\,dx+ \eta_\varepsilon^p\int_{Q_\varepsilon}|\partial_{x_3} R_p^\varepsilon(\varphi)|^p\,dx\\
\\
\displaystyle
 \leq C\left(\int_{Q_\varepsilon}|\varphi|^p\,dx+ \varepsilon^p\int_{Q_\varepsilon}|D_{x'} \varphi|^p\,dx+\eta_\varepsilon^p\int_{Q_\varepsilon}|\partial_{x_3} \varphi|^p\,dx\right)\,.
 \end{array}
 $$
 Obviously $R_p^\varepsilon(\varphi)$ lies in $W^{1,p}_0(\Omega_\varepsilon)^3$ and is equal to $\varphi$ if $\varphi$ is zero on $Q_\varepsilon\setminus\Omega_\varepsilon$, so we get the estimates in the third item. Moreover, the second item is obvious from (\ref{lema_1_bayada_estim})$_2$ and the definition of $R_p^\varepsilon$.\\
\end{proof}
\begin{lemma}\label{lemm_R_tilde} Setting $\widetilde R_p^\varepsilon(\widetilde \varphi)=R_p^\varepsilon(\varphi)$ for any $\widetilde\varphi$ in $W^{1,p}_0(\Omega)^3$, $1<p<+\infty$, where $\tilde \varphi(x',y_3)=\varphi(x',\eta_\varepsilon y_3)$ and $R_p^\varepsilon$ is defined in Lemma \ref{Operator}, we have the following estimates
\begin{itemize}
\item[i)] if $\eta_\varepsilon\approx \varepsilon$, with $\eta_\varepsilon/\varepsilon\to \lambda$, $0<\lambda<+\infty$ or $\eta_\varepsilon\ll \varepsilon$,  then
$$\begin{array}{l}\displaystyle
\|\widetilde R_p^\varepsilon(\tilde \varphi)\|_{L^p(\widetilde\Omega_\varepsilon)^3}\leq C\|\tilde \varphi\|_{W_0^{1,p}(\Omega)^3}\,,\\
\\
\displaystyle \|D_{\eta_\varepsilon} \widetilde R_p^\varepsilon(\tilde \varphi)\|_{L^p(\widetilde\Omega_\varepsilon)^{3\times 3}}\leq C{1 \over \eta_\varepsilon}\|\tilde \varphi\|_{W_0^{1,p}(\Omega)^3}\,,\\
\end{array}$$

\item[ii)] if $\eta_\varepsilon\gg \varepsilon$, then
$$\begin{array}{l}\displaystyle
\|\widetilde R_p^\varepsilon(\tilde \varphi)\|_{L^p(\widetilde\Omega_\varepsilon)^3}\leq C\|\tilde \varphi\|_{W_0^{1,p}(\Omega)^3}\,,\\
\\
\displaystyle \|D_{\eta_\varepsilon} \widetilde R_p^\varepsilon(\tilde \varphi)\|_{L^p(\widetilde\Omega_\varepsilon)^{3\times 3}}\leq C{1 \over \varepsilon}\|\tilde \varphi\|_{W_0^{1,p}(\Omega)^3}\,.\end{array}$$
\end{itemize}
\end{lemma}
\begin{proof} Considering the change of variables given in (\ref{dilatacion})  and the estimates given in Lemma \ref{Operator}, we obtain
$$\begin{array}{l}\displaystyle
\|\widetilde R_p^\varepsilon(\tilde \varphi)\|_{L^p(\widetilde\Omega_\varepsilon)^3}\leq C\left(\|\tilde \varphi\|_{L^p(\Omega)^3}+\,\varepsilon \|D_{x'}\tilde \varphi\|_{L^p(\Omega)^{3\times 2}}+\, \|\partial_{y_3}\tilde \varphi\|_{L^p(\Omega)^{3}}\right)\,,\\
\\
\displaystyle \|D_{x'} \widetilde R_p^\varepsilon(\tilde \varphi)\|_{L^p(\widetilde\Omega_\varepsilon)^{3\times 2}}\leq C\left({1 \over \varepsilon}\|\tilde \varphi\|_{L^p(\Omega)^3}+\, \|D_{x'}\tilde \varphi\|_{L^p(\Omega)^{3\times 2}}+\,{1 \over \varepsilon} \|\partial_{y_3}\tilde \varphi\|_{L^p(\Omega)^{3}}\right)\,,\\
\\
\displaystyle \|\partial_{y_3} \widetilde R_p^\varepsilon(\tilde \varphi)\|_{L^p(\widetilde\Omega_\varepsilon)^{3}}\leq C\left(\|\tilde \varphi\|_{L^p(\Omega)^3}+\, \varepsilon \|D_{x'}\tilde \varphi\|_{L^p(\Omega)^{3\times 2}}+ \|\partial_{y_3}\tilde \varphi\|_{L^p(\Omega)^{3}}\right)\,.
\end{array}$$
Taking into account that $\varepsilon,\eta_\varepsilon \ll 1$ and the relation between $\varepsilon$ and $\eta_\varepsilon$, we have the desired result.
\end{proof}

It is then possible, to use the classical Tartar extension of the pressure.

\begin{lemma}\label{estOmega} Assume that $9/5\leq p< +\infty$. 
There exists a constant $C$ independent of $\varepsilon$, such that the extension $(\tilde{v}_{\varepsilon}, \tilde{P}_{\varepsilon})\in W_0^{1,p}(\Omega)^3\times L^{p^\prime}_0(\Omega)$ of a solution $(\tilde{u}_{\varepsilon}, \tilde{p}_{\varepsilon})$ of (\ref{2})-(\ref{no-slip_tilde}) satisfies
\begin{equation}\label{estim_velocidad}
\left\Vert \tilde{v}_{\varepsilon}\right\Vert_{L^p(\Omega)^3}\leq C\eta_\varepsilon^{p\over p-1},\quad \left\Vert \mathbb{D}_{\eta_\varepsilon}\left[\tilde{v}_{\varepsilon}\right]\right\Vert_{L^p(\Omega)^{{3\times3}}}\leq C\eta_\varepsilon^{1\over p-1},
\end{equation}
\begin{equation}\label{estim_velocidad2}
\left\Vert D_{\eta_\varepsilon}\tilde{v}_{\varepsilon}\right\Vert_{L^p(\Omega)^{{3\times3}}}\leq C\eta_\varepsilon^{1\over p-1}.
\end{equation}
For the cases $\eta_\varepsilon\approx \varepsilon$,  with $\eta_\varepsilon/\varepsilon\to \lambda$, $0<\lambda<+\infty$, or $\eta_\varepsilon\ll \varepsilon$, we have
\begin{equation}\label{esti_P}
\left\Vert \tilde{P}_{\varepsilon} \right\Vert_{L^{p^\prime}_0(\Omega)}\leq C,
\end{equation}
and for the case $\eta_\varepsilon\gg \varepsilon$, we have
\begin{equation}\label{esti_P2Omega}
\left\Vert \tilde{P}_{\varepsilon} \right\Vert_{L^{p^\prime}_0(\Omega^-)}\leq C.
\end{equation}
\end{lemma}

\begin{proof} We first estimate the velocity. Taking into account Lemma \ref{Lemma_a3}, it is clear that, after extension, (\ref{estim_velocidad})-(\ref{estim_velocidad2}) hold.

The mapping $R_p^\varepsilon$ defined in Lemma \ref{Operator} allows us to extend the pressure $p_\varepsilon$ to $Q_\varepsilon$ introducing $F_\varepsilon$ in $W^{-1,p'}(Q_\varepsilon)^3$:
\begin{equation}\label{F}\langle F_\varepsilon, \varphi\rangle_{Q_\varepsilon}=\langle \nabla p_\varepsilon, R_p^\varepsilon (\varphi)\rangle_{\Omega_\varepsilon}\,,\quad \hbox{for any }\varphi\in W^{1,p}_0(Q_\varepsilon)^3\,.
\end{equation}
We calcule the right hand side of (\ref{F}) by using (\ref{1}) and we have
\begin{equation}\label{equality_duality}
\begin{array}{rl}
\displaystyle
\left\langle F_{\varepsilon},\varphi\right\rangle_{Q_\varepsilon}=&\displaystyle
-\mu\int_{\Omega_\varepsilon}S(\mathbb{D}\left[{u}_\varepsilon\right]) : D R_p^{\varepsilon}(\varphi)\,dx\\
&\displaystyle+\int_{\Omega_\varepsilon} f\cdot R_p^{\varepsilon}(\varphi)\,dx - \int_{\Omega_\varepsilon} (u_\varepsilon\cdot \nabla) u_\varepsilon\,R_p^{\varepsilon}(\varphi)\,dx\,.
\end{array}\end{equation}
Moreover, ${\rm div} \varphi=0$ implies 
$$\left\langle F_{\varepsilon},\varphi\right\rangle_{Q_\varepsilon}=0\,,$$
and the DeRham theorem gives the existence of $P_\varepsilon$ in $L^{p'}_0(Q_\varepsilon)$ with $F_\varepsilon=\nabla P_\varepsilon$.

We get for any $\tilde \varphi\in W^{1,p}_0(\Omega)^3$, using the change of variables (\ref{dilatacion}),
$$\begin{array}{rl}\displaystyle\langle \nabla_{\eta_\varepsilon}\tilde P_\varepsilon, \tilde\varphi\rangle_{\Omega}&\displaystyle
=-\int_{\Omega}\tilde P_\varepsilon\,{\rm div}_{\eta_\varepsilon}\,\tilde\varphi\,dx'dy_3\\
&\displaystyle
=-\eta_\varepsilon^{-1}\int_{Q_\varepsilon}P_\varepsilon\,{\rm div}\,\varphi\,dx=\eta_\varepsilon^{-1}\langle \nabla P_\varepsilon, \varphi\rangle_{Q_\varepsilon}\,.
\end{array}$$
Then, using the identification (\ref{equality_duality}) of $F_\varepsilon$,
$$\begin{array}{rl}\displaystyle\langle \nabla_{\eta_\varepsilon}\tilde P_\varepsilon, \tilde\varphi\rangle_{\Omega}&\displaystyle
=\eta_\varepsilon^{-1}\Big(-\mu \int_{\Omega_\varepsilon}S(\mathbb{D}\left[{u}_\varepsilon\right]) : D R_p^{\varepsilon}(\varphi)\,dx\\
&\displaystyle
+\int_{\Omega_\varepsilon} f\cdot R_p^{\varepsilon}(\varphi)\,dx - \int_{\Omega_\varepsilon}(u_\varepsilon\cdot \nabla) u_\varepsilon R_p^{\varepsilon}(\varphi)\,dx\Big)\,,
\end{array}$$
and applying the change of variables (\ref{dilatacion}),
\begin{equation}\label{MM}\begin{array}{rl}\displaystyle\langle \nabla_{\eta_\varepsilon}\tilde P_\varepsilon, \tilde\varphi\rangle_{\Omega}=&\!\!\!\!\displaystyle
-\mu \int_{\widetilde \Omega_\varepsilon}S(\mathbb{D}_{\eta_\varepsilon}\left[\tilde {u}_\varepsilon\right]) : D_{\eta_\varepsilon} \tilde R_p^{\varepsilon}(\tilde \varphi)\,dx'dy_3 \\
\\
&\displaystyle\!\!\!\! \!\!\!\!\!\!\!\!\!\!\!\!+\int_{\widetilde \Omega_\varepsilon} f\cdot \tilde R_p^{\varepsilon}(\tilde \varphi)\,dx'dy_3- \int_{\widetilde \Omega_\varepsilon}(\tilde u_\varepsilon\cdot \nabla_{\eta_\varepsilon}) \tilde u_\varepsilon \tilde R_p^{\varepsilon}(\tilde \varphi)\,dx'dy_3\,.
\end{array}
\end{equation}

Now, we estimate the right-hand side of (\ref{MM}). First, we consider $\eta_\varepsilon\approx \varepsilon$ or $\eta_\varepsilon\ll \varepsilon$.

Using (\ref{a3}) and Lemma \ref{lemm_R_tilde}-(i), we get 
$$\begin{array}{rl}\displaystyle
\left|\mu\int_{\widetilde \Omega_\varepsilon}S(\mathbb{D}_{\eta_\varepsilon}\left[\tilde {u}_\varepsilon\right]) : D_{\eta_\varepsilon} \tilde R_p^{\varepsilon}(\tilde \varphi)\,dx'dy_3\right|\leq &\displaystyle
  C\eta_\varepsilon \|D_{\eta_\varepsilon}\tilde R_p^{\varepsilon}(\tilde \varphi)\|_{L^p(\widetilde\Omega_\varepsilon)^{3\times 3}}\\
  \leq &\displaystyle 
   C\|\tilde\varphi\|_{W^{1,p}_0(\Omega)^3}\,,\\
  \\
  \displaystyle \left|\int_{\widetilde \Omega_\varepsilon} f\cdot \tilde R_p^{\varepsilon}(\tilde \varphi)\,dx'dy_3\right|\leq &\displaystyle C\| \tilde R_p^{\varepsilon}(\tilde \varphi)\|_{L^p(\widetilde\Omega_\varepsilon)^3}\\
 \leq&\displaystyle C\|\tilde\varphi\|_{W^{1,p}_0(\Omega)^3}\,.
  \end{array}
  $$
  
  For the intertial terms, we proceed as in the proof of Lemma \ref{lem_estim_nabla_pressure}. We have
$$\begin{array}{l}\displaystyle  \left|\int_{\widetilde \Omega_\varepsilon}(\tilde u_\varepsilon\cdot \nabla_{\eta_\varepsilon}) \tilde u_\varepsilon \tilde R_p^{\varepsilon}(\tilde \varphi)\,dx'dy_3\right|\leq  \displaystyle
\left|\int_{\widetilde \Omega_\varepsilon} \tilde u_\varepsilon\tilde \otimes \tilde u_\varepsilon:D_{x'}  \tilde R_p^{\varepsilon}(\tilde \varphi)\,dx'dy_3\right|\\
\\
\displaystyle\qquad+{1\over \eta_\varepsilon}\left|\int_{\widetilde \Omega_\varepsilon}\partial_{y_3} \tilde u_{\varepsilon,3}\tilde  u_\varepsilon  \tilde R_p^{\varepsilon}(\tilde \varphi) \,dx'dy_3+ \int_{\widetilde \Omega_\varepsilon}\tilde u_{\varepsilon,3}\partial_{y_3} \tilde u_\varepsilon\,  \tilde R_p^{\varepsilon}(\tilde\varphi)\,dx'dy_3\right|\,.
\end{array}$$
Proceeding exactly as the proof of Lemma \ref{lem_estim_nabla_pressure} and taking into account Lemma \ref{lemm_R_tilde}-(i), we obtain for $9/5\leq p< 3$
$$\left|\int_{\widetilde \Omega_\varepsilon} (\tilde u_\varepsilon\cdot \nabla_{\eta_\varepsilon})\tilde u_\varepsilon\,\tilde R_p^{\varepsilon}(\tilde \varphi)\, dx'dy_3\right|\leq  C\eta_\varepsilon^{3-p\over p-1}\|\tilde \varphi\|_{W^{1,p}_0
(\Omega)^3}\,,$$
and for $p\ge 3$
$$\left|\int_{\widetilde \Omega_\varepsilon} (\tilde u_\varepsilon\cdot \nabla_{\eta_\varepsilon})\tilde u_\varepsilon\,\tilde R_p^{\varepsilon}(\tilde \varphi)\, dx'dy_3\right|\leq  C\eta_\varepsilon^{p+1\over p-1}\|\tilde \varphi\|_{W^{1,p}_0
(\Omega)^3}\,.$$
Then, for $9/5\leq p<+\infty$, we can deduce
$$\left|\int_{\widetilde \Omega_\varepsilon} (\tilde u_\varepsilon\cdot \nabla_{\eta_\varepsilon})\tilde u_\varepsilon\,\tilde R_p^{\varepsilon}(\tilde \varphi)\, dx'dy_3\right|\leq  C\|\tilde \varphi\|_{W^{1,p}_0
(\Omega)^3}\,,$$
and from (\ref{MM}), we obtain
\begin{equation}\label{nabla_p_estim}\|\nabla_{\eta_\varepsilon}\tilde P_\varepsilon\|_{L^{p'}_0(\Omega)^3}\leq C,
\end{equation}
which implies (\ref{esti_P}).

In the case $\eta_\varepsilon\gg \varepsilon$, due to the highly oscillating boundary, we will obtain that the velocity will be zero in $\omega\times (h_{\rm min},h_{\rm max})$ (see (\ref{convtilde1_super_2}) for more details), and so we will obtain an effective problem posed in $\Omega^-$. 

Therefore, reproducing Lemma \ref{lem_estim_nabla_pressure} by considering $\tilde \varphi \in W^{1,p}_0(\Omega^-)$, and taking into account that $\tilde R_p^\varepsilon(\tilde \varphi )=\tilde \varphi$ in $\Omega^-$, we deduce that 
$$\|\nabla_{\eta_\varepsilon}\tilde P_\varepsilon\|_{L^{p'}_0(\Omega^-)^3}\leq C,$$
which implies (\ref{esti_P2Omega}).
\end{proof}

\subsection{Adaptation of the Unfolding Method}
The change of variable (\ref{dilatacion}) does not provide the information we need about the behavior of $\tilde{u}_{\varepsilon}$ in the microstructure associated to $\widetilde{\Omega}_{\varepsilon}$. To solve this difficulty, we introduce an adaptation of the unfolding method (see \cite{arbogast, Ciora} for more details). For this purpose, given $\tilde{u}_{\varepsilon}\in W_0^{1,p}(\widetilde \Omega_\varepsilon)^3$ a solution of the rescaled system (\ref{2})-(\ref{no-slip_tilde}), we define $\hat{u}_{\varepsilon}$ by
\begin{eqnarray}\label{uhat}
\hat{u}_{\varepsilon}(x^{\prime},y)=\tilde{u}_{\varepsilon}\left( {\varepsilon}\kappa\left(\frac{x^{\prime}}{{\varepsilon}} \right)+{\varepsilon}y^{\prime},y_3 \right),\text{\ \ a.e. \ }(x^{\prime},y)\in \omega\times Y,
\end{eqnarray}
where the function $\kappa$ is defined as follows; for $k'\in \mathbb{Z}^2$, we define $\kappa: \mathbb{R}^2\to \mathbb{Z}^2$ by
$$
\kappa(x')=k' \Longleftrightarrow x'\in Y'_{k',1}\,.
$$
Remark that $\kappa$ is well defined up to a set of zero measure in $\mathbb{R}^2$ (the set $\cup_{k'\in \mathbb{Z}^2}\partial Y'_{k',1}$). Moreover, for every $\varepsilon>0$, we have
$$\kappa\left({x'\over \varepsilon}\right)=k'\Longleftrightarrow x'\in Y'_{k',\varepsilon}\,.$$

In the same sense,  given the extension of the pressure $\tilde P_\varepsilon\in L^{p'}_0(\Omega)$, we define $\hat P_\varepsilon$  by
\begin{eqnarray}\label{Phat}
\hat{P}_{\varepsilon}(x^{\prime},y)=\tilde{P}_{\varepsilon}\left( {\varepsilon}\kappa\left(\frac{x^{\prime}}{{\varepsilon}} \right)+{\varepsilon}y^{\prime},y_3 \right),\text{\ \ a.e. \ }(x^{\prime},y)\in \omega\times \Pi.
\end{eqnarray}

\begin{remark}\label{remarkCV}
For $k^{\prime}\in T_{\varepsilon}$, the restrictions of $\hat{u}_{\varepsilon}$  to $Y^{\prime}_{k^{\prime},{\varepsilon}}\times Y$ and $\hat{P}_{\varepsilon}$ to  $Y^{\prime}_{k^{\prime},{\varepsilon}}\times \Pi$ do not depend on $x^{\prime}$, whereas as a function of $y$ it is obtained from $(\tilde{u}_{\varepsilon}, \tilde{P}_{\varepsilon})$ by using the change of variables 
\begin{equation}\label{CV}
y^{\prime}=\frac{x^{\prime}-{\varepsilon}k^{\prime}}{{\varepsilon}}, 
\end{equation}
which transforms $Y_{k^{\prime},{\varepsilon}}$ into $Y$ and  $\widetilde Q_{k',\varepsilon}$ into $\Pi$, respectively.
\end{remark}
Let us obtain some estimates for the sequences $(\hat{u}_{\varepsilon}, \hat{P}_{\varepsilon})$. 
\begin{lemma}\label{estCV}
%%%%%%%%%%%%%%%%%%%%%
Assume that $9/5\leq p< +\infty$. There exists a constant $C$ independent of $\varepsilon$, such that $(\hat{u}_{\varepsilon}, \hat{P}_{\varepsilon})$ defined by (\ref{uhat})-(\ref{Phat}) satisfies
\begin{equation}\label{estim_velocidad_gorro1}
\left\Vert \mathbb{D}_{y'}\!\left[\hat{u}_{\varepsilon}\right]\right\Vert_{L^p(\omega\times Y)^{{3\times 2}}}\!\leq\! C\varepsilon \eta_{\varepsilon}^{1 \over p-1},\quad \left\Vert \partial_{y_3}\!\left[\hat{u}_{\varepsilon}\right]\right\Vert_{L^p(\omega\times Y)^{{3}}}\!\leq\! C \eta_{\varepsilon}^{p \over p-1},\end{equation}
\begin{equation}\label{estim_velocidad_gorro1SIMPLE}
\left\Vert D_{y'}\hat{u}_{\varepsilon}\right\Vert_{L^p(\omega\times Y)^{{3\times 2}}}\!\leq\!  C\varepsilon \eta_{\varepsilon}^{1 \over p-1}, \quad \left\Vert \partial_{y_3}\hat{u}_{\varepsilon}\right\Vert_{L^p(\omega\times Y)^{{3}}}\!\leq\! C\eta_{\varepsilon}^{p \over p-1}
\end{equation}
\begin{equation}\label{estim_velocidad_gorro2}
\left\Vert \hat{u}_{\varepsilon}\right\Vert_{L^p(\omega\times Y)^3}\leq C\eta_{\varepsilon}^{\frac{p}{p-1}}\,.
\end{equation}
For the cases $\eta_\varepsilon\approx \varepsilon$,  with $\eta_\varepsilon/\varepsilon\to \lambda$, $0<\lambda<+\infty$, or $\eta_\varepsilon\ll \varepsilon$, we have
\begin{equation}\label{esti_gorro}
\left\Vert \hat{P}_{\varepsilon} \right\Vert_{L^{p'}(\omega\times \Pi)}\leq C.
\end{equation}
\end{lemma}
\begin{proof}
Let us obtain some estimates for the sequence $\hat{u}_{\varepsilon}$ defined by (\ref{uhat}). We obtain
$$\begin{array}{l}
\displaystyle \int_{\omega\times Y}\left\vert \mathbb{D}_{y^{\prime}}\left[\hat{u}_{\varepsilon}(x^{\prime},y)\right]\right\vert^pdx^{\prime}dy=\displaystyle\sum_{k^{\prime}\in T_{\varepsilon}}\int_{Y^{\prime}_{k^{\prime},{\varepsilon}}}\int_{Y}\left\vert \mathbb{D}_{y^{\prime}}\left[\hat{u}_{\varepsilon}(x^{\prime},y)\right]\right\vert^pdx^{\prime}dy\\
\qquad\qquad=\displaystyle\sum_{k^{\prime}\in T_{\varepsilon}}\int_{Y^{\prime}_{k^{\prime},{\varepsilon}}}\int_{Y'}\int_0^{h(y')}\left\vert \mathbb{D}_{y^{\prime}}\left[\tilde{u}_{\varepsilon}({\varepsilon}k^{\prime}+{\varepsilon}y^{\prime},y_3)\right]\right\vert^pdx^{\prime}dy'dy_3.
\end{array}
$$
We observe that $\tilde{u}_{\varepsilon}$ does not depend on $x^{\prime}$, then we can deduce
$$\begin{array}{l}\displaystyle
\int_{\omega\times Y}\left\vert \mathbb{D}_{y^{\prime}}\left[\hat{u}_{\varepsilon}(x^{\prime},y)\right]\right\vert^pdx^{\prime}dy\\
\displaystyle \qquad= {\varepsilon}^2\displaystyle\sum_{k^{\prime}\in T_{\varepsilon}}\int_{Y'}\int_0^{h(y')}\left\vert \mathbb{D}_{y^{\prime}}\left[\tilde{u}_{\varepsilon}({\varepsilon}k^{\prime}+{\varepsilon}y^{\prime},y_3)\right]\right\vert^pdy'dy_3.
\end{array}
$$
By the change of variables (\ref{CV}) and by the $Y'$-periodicity of $h$, we obtain
$$\begin{array}{l}
\displaystyle \int_{\omega\times Y}\left\vert \mathbb{D}_{y^{\prime}}\left[\hat{u}_{\varepsilon}(x^{\prime},y)\right]\right\vert^pdx^{\prime}dy\\
\displaystyle
\qquad={\varepsilon}^p\sum_{k^{\prime}\in T_{\varepsilon}}\int_{Y^{\prime}_{k^{\prime},{\varepsilon}}}\int_0^{h({x'\over \varepsilon}-k')}\left\vert \mathbb{D}_{x^{\prime}}\left[\tilde{u}_{\varepsilon}(x^{\prime},y_3)\right]\right\vert^pdx^{\prime}dy_3\\
\qquad\displaystyle={\varepsilon}^p\displaystyle\sum_{k^{\prime}\in T_{\varepsilon}}\int_{Y^{\prime}_{k^{\prime},{\varepsilon}}}\int_0^{h({x'\over \varepsilon})}\left\vert \mathbb{D}_{x^{\prime}}\left[\tilde{u}_{\varepsilon}(x^{\prime},y_3)\right]\right\vert^pdx^{\prime}dy_3\\
\qquad\displaystyle= {\varepsilon}^p\int_{\widetilde \Omega_\varepsilon}\left\vert \mathbb{D}_{x^{\prime}}\left[\tilde{u}_{\varepsilon}(x^{\prime},y_3)\right]\right\vert^pdx^{\prime}dy_3.
\end{array}
$$
Taking into account the second estimate in (\ref{a3}), we get the first estimate in (\ref{estim_velocidad_gorro1}).

Similarly, using Remark \ref{remarkCV} and definition (\ref{uhat}), we have
\begin{eqnarray*}
\int_{\omega\times Y}\left\vert \partial_{y_3}\left[\hat{u}_{\varepsilon}(x^{\prime},y)\right]\right\vert^pdx^{\prime}dy\leq {\varepsilon}^2\displaystyle\sum_{k^{\prime}\in T_{\varepsilon}}\int_{Y}\left\vert \partial_{y_3}\left[\tilde{u}_{\varepsilon}({\varepsilon}k^{\prime}+{\varepsilon}y^{\prime},y_3)\right]\right\vert^pdy.
\end{eqnarray*}
By the change of variables (\ref{CV}) and the second estimate in (\ref{a3}),  we obtain
\begin{eqnarray*}
\int_{\omega\times Y}\left\vert \partial_{y_3}\left[\hat{u}_{\varepsilon}(x^{\prime},y)\right]\right\vert^pdx^{\prime}dy\leq \int_{\widetilde \Omega_\varepsilon}\left\vert \partial_{y_3}\left[\tilde{u}_{\varepsilon}(x^{\prime},y_3)\right]\right\vert^pdx^{\prime}dy_3
\leq C \eta_{\varepsilon}^{\frac{p^2}{p-1}},
\end{eqnarray*}
so the second estimate in (\ref{estim_velocidad_gorro1}) is proved. Consequently, from classical Korn's inequality, we also have (\ref{estim_velocidad_gorro1SIMPLE}). 

Similarly, using the definition (\ref{uhat}), the change of variables (\ref{CV}) and the first estimate in  (\ref{a3}), we have
\begin{eqnarray*}
\int_{\omega\times Y}\left\vert \hat{u}_{\varepsilon}(x^{\prime},y)\right\vert^pdx^{\prime}dy \leq C\eta_{\varepsilon}^{\frac{p^2}{p-1}},
\end{eqnarray*}
and (\ref{estim_velocidad_gorro2}) holds. 

Finally, let us obtain some estimates for the sequence $\hat{P}_{\varepsilon}$ defined by (\ref{Phat}).
We observe that using the definition (\ref{Phat}) of $\hat{P}_{\varepsilon}$, we obtain
\begin{eqnarray*}
\int_{\omega\times \Pi}\left\vert \hat{P}_{\varepsilon}(x^{\prime},y)\right\vert^{p^\prime}dx^{\prime}dy\leq\displaystyle\sum_{k^{\prime}\in T_{\varepsilon}}\int_{Y^{\prime}_{k^{\prime},{\varepsilon}}}\int_{Y'}\int_0^{h_{\rm max}}\left\vert \tilde{P}_{\varepsilon}({\varepsilon}k^{\prime}+{\varepsilon}y^{\prime},y_3)\right\vert^{p^\prime}dx^{\prime}dy.
\end{eqnarray*}
We observe that $\tilde{P}_{\varepsilon}$ does not depend on $x^{\prime}$, then we can deduce
\begin{eqnarray*}
\int_{\omega\times \Pi}\left\vert \hat{P}_{\varepsilon}(x^{\prime},y)\right\vert^{p^\prime}dx^{\prime}dy\leq {\varepsilon}^2\displaystyle\sum_{k^{\prime}\in T_{\varepsilon}}\int_{Y'}\int_0^{h_{\rm max}}\left\vert \tilde{P}_{\varepsilon}({\varepsilon}k^{\prime}+{\varepsilon}y^{\prime},y_3)\right\vert^{p^\prime}dy'dy_3.
\end{eqnarray*}
By the change of variables (\ref{CV}), we obtain
\begin{eqnarray*}
\int_{\omega\times \Pi}\left\vert \hat{P}_{\varepsilon}(x^{\prime},y)\right\vert^{p^\prime}dx^{\prime}dy\leq \int_{\Omega}\left\vert \tilde{P}_{\varepsilon}(x^{\prime},y_3)\right\vert^{p^\prime}dx^{\prime}dy_3.
\end{eqnarray*}
Taking into account (\ref{esti_P}), we have (\ref{esti_gorro}).
\end{proof}

\section{Some compactness results}\label{S5}
In this section we obtain some compactness results about the behavior of the sequences $(\tilde{v}_{\varepsilon}, \tilde{P}_{\varepsilon})$ and $(\hat{u}_{\varepsilon}, \hat{P}_{\varepsilon})$ satisfying {\it a priori} estimates given in Lemma \ref{estOmega} and Lemma \ref{estCV} respectively. We obtain different behaviors depending on the magnitude $\eta_{\varepsilon}$ with respect to $\varepsilon$.

Let us start giving a convergence result for the pressure $\tilde{P}_{\varepsilon}$.
\begin{lemma} \label{pseudo_lemma1}Assume $9/5\leq p<+\infty $. For the cases $\eta_\varepsilon\approx \varepsilon$,  with $\eta_\varepsilon/\varepsilon\to \lambda$, $0<\lambda<+\infty$, or $\eta_\varepsilon\ll \varepsilon$, for a subsequence of $\varepsilon$ still denote by $\varepsilon$, there exists $\tilde{P}\in L^{p^\prime}_0(\Omega)$, independent of $y_3$, such that
\begin{equation}\label{convPtilde1}
\tilde{P}_{\varepsilon}\rightharpoonup \tilde{P}\text{\ in \ }L_0^{p^\prime}(\Omega),
\end{equation}
and for the case $\eta_\varepsilon\gg \varepsilon$, 
\begin{equation}\label{convPtilde1_menos}
\tilde{P}_{\varepsilon}\rightharpoonup \tilde{P}\text{\ in \ }L_0^{p^\prime}(\Omega^-).
\end{equation}
\end{lemma}
\begin{proof}
Estimate (\ref{esti_P}) implies, up to a subsequence, the existence of $\tilde{P}\in L^{p^\prime}_0(\Omega)$ such that (\ref{convPtilde1}) holds. Also, from (\ref{nabla_p_estim}), by noting that $\partial_{y_3}\tilde P_\varepsilon/\eta_\varepsilon$ also converges weakly  in $W^{-1,p'}(\Omega)$, we obtain $\partial_{y_3} \tilde P=0$. Analogously, we obtain (\ref{convPtilde1_menos}). \par
\end{proof}

We will give a convergence result for $\tilde{v}_{\varepsilon}$. 
\begin{lemma}\label{LemmaConvergenceUtilde} Assume that $9/5\leq p<+\infty$. For a subsequence of $\varepsilon$ still denote by $\varepsilon$,  there exists $\tilde{v}\in W^{1,p}(0,h_{\rm max};L^p(\omega)^3)$ where $\tilde v_3=0$, and $\tilde{v}(x',0)=\tilde v(x',h_{\rm max})=0$, such that
\begin{equation}\label{convtilde1}
\eta_\varepsilon^{-{p \over p-1}}\tilde{v}_{\varepsilon}\rightharpoonup (\tilde{v}^\prime,0)\text{\ in \ }W^{1,p}(0,h_{\rm max};L^p(\omega)^3)\,,
\end{equation}
and 
\begin{equation}\label{divmacro_tilde1}
\left\{\begin{array}{l}\displaystyle
{\rm div}_{x^{\prime}}\left( \int_0^{h_{\rm max}} \tilde{v}^{\prime}(x^{\prime},y_3)dy_3 \right)=0 \text{\ in \ }\omega,\\
\displaystyle \left( \int_0^{h_{\rm max}} \tilde{v}^{\prime}(x^{\prime},y_3)dy_3 \right)\cdot n=0\text{\ on \ }\partial \omega.
\end{array}\right.
\end{equation}
Moreover, for the case $\eta_\varepsilon\gg \varepsilon$, we have
\begin{equation}\label{convtilde1_super_2}
\eta_\varepsilon^{-{p \over p-1}}\tilde{v}_{\varepsilon}\rightharpoonup 0\text{\ in \ }W^{1,p}(h_{\rm min},h_{\rm max};L^p(\omega)^3)\,,
\end{equation}
and
\begin{equation}\label{divmacro_tilde1_super}
\left\{\begin{array}{l}\displaystyle
{\rm div}_{x^{\prime}}\left( \int_0^{h_{\rm min}} \tilde{v}^{\prime}(x^{\prime},y_3)dy_3 \right)=0 \text{\ in \ }\omega,\\
\displaystyle \left( \int_0^{h_{\rm min}} \tilde{v}^{\prime}(x^{\prime},y_3)dy_3 \right)\cdot n=0\text{\ on \ }\partial \omega.
\end{array}\right.\end{equation}

\end{lemma}

\begin{proof} The estimates (\ref{estim_velocidad})-(\ref{estim_velocidad2}) read 
\begin{equation*}\label{estimate_case1-tilde}
\left\Vert \tilde{v}_{\varepsilon}\right\Vert_{L^p(\Omega)^3}\leq C \eta_\varepsilon^{p\over p-1}, \quad \left\Vert D_{x^\prime}\tilde{v}_{\varepsilon}\right\Vert_{L^p(\Omega)^{{3\times2}}}\leq C \eta_\varepsilon^{1\over p-1}, \quad \left\Vert \partial_{y_3}\tilde{v}_{\varepsilon}\right\Vert_{L^p(\Omega)^{{3}}}\leq C \eta_\varepsilon^{p\over p-1}.
\end{equation*}
The above estimates imply the existence $\tilde{v}\in W^{1,p}(0,h_{\rm max};L^p(\omega)^3)$, such that, up to a subsequence, we have
\begin{equation}\label{conv_lema_1}
\eta_\varepsilon^{-{p\over p-1}}\tilde{v}_{\varepsilon}\rightharpoonup \tilde{v}\text{\ in \ }W^{1,p}(0,h_{\rm max};L^p(\omega)^3),
\end{equation}
which implies
\begin{equation}\label{conv_div_123}
\eta_\varepsilon^{-{p\over p-1}}{\rm div}_{x^\prime}\tilde{v}_{\varepsilon}^{\prime}\rightharpoonup {\rm div}_{x^\prime}\tilde{v}^{\prime}\text{\ in \ }W^{1,p}(0,h_{\rm max};W^{-1,p'}(\omega)).
\end{equation}
Since ${\rm div}_{\eta_\varepsilon}\tilde{v}_{\varepsilon}=0$ in $\Omega$, multiplying by $\eta_\varepsilon^{-{p\over p-1}}$ we obtain
\begin{equation*}\label{div_important_Tilde} 
\eta_\varepsilon^{-{p\over p-1}}{\rm div}_{x^\prime}\tilde{v}_{\varepsilon}^\prime +\eta_\varepsilon^{-{2p-1 \over p-1}}\partial_{y_3}\tilde{v}_{\varepsilon,3}=0,\quad \text{\ in\ } \Omega,
\end{equation*}
which, combined with (\ref{conv_div_123}), implies that  $\eta_\varepsilon^{-{2p-1 \over p-1}}\partial_{y_3}\tilde{v}_{\varepsilon,3}$ is bounded in $W^{1,p}(0,h_{\rm max};W^{-1,p'}(\omega))$. This implies that $\eta_\varepsilon^{-{p \over p-1}}\partial_{y_3}\tilde{v}_{\varepsilon,3}$ tends to zero in $W^{1,p}(0,h_{\rm max};W^{-1,p'}(\omega))$. Also, from (\ref{conv_lema_1}), we have that $\eta_\varepsilon^{-{p \over p-1}}\partial_{y_3}\tilde{v}_{\varepsilon,3}$ tends to $\partial_{y_3}\tilde{v}_{3}$ in $L^p(\Omega)$. From the uniqueness of the limit, we have that $\partial_{y_3}\tilde{v}_3=0$, which implies that $\tilde v_3$ does not depend on $y_3$. Moreover, the continuity of the trace applications from the space of functions $v$ such that $\|\tilde v\|_{L^{p}}$ and $\|\partial_{y_3}\tilde  v\|_{L^p}$ to $L^p(\Sigma)$ and to $L^p(\omega\times\{0\})$ implies $\tilde v=0$ on $\Sigma$ and $\omega\times\{0\}$. This together to $\partial_{y_3} \tilde v_3=0$ implies that $\tilde v_3=0$. 
\\

Finally,  we prove (\ref{divmacro_tilde1}). To do this, we consider $\varphi\in C_c^1(\omega)$ as test function in ${\rm div}_{\eta_\varepsilon}\tilde v_{\varepsilon}=0$ in $\Omega$, which multiplying by $\eta_{\varepsilon}^{-{p\over p-1}}$ gives
$$
\int_{\Omega} {\rm div}_{x'}\tilde v'_{\varepsilon}\, \varphi(x')\, dx'dy_3=0.
$$
From convergence (\ref{convtilde1}), we get (\ref{divmacro_tilde1}).

Finally, for the case $\eta_\varepsilon\gg \varepsilon$, following Theorem 5.2. in Chambat {\it et al} \cite{CBF}, we obtain (\ref{convtilde1_super_2}). As consequence, this together with  (\ref{divmacro_tilde1}) gives (\ref{divmacro_tilde1_super}).

\end{proof}

Now, we give a convergence result for the pressure $\hat{P}_{\varepsilon}$.
\begin{lemma} Assume that $9/5\leq p<+\infty$. For the cases $\eta_\varepsilon\approx \varepsilon$,  with $\eta_\varepsilon/\varepsilon\to \lambda$, $0<\lambda<+\infty$, or $\eta_\varepsilon\ll \varepsilon$, for a subsequence of $\varepsilon$ still denote by $\varepsilon$ there exists $\hat{P}\in L_0^{p^\prime}(\omega \times \Pi)$ such that
\begin{equation}\label{convPgorro1}
\hat{P}_{\varepsilon}\rightharpoonup \hat{P}\text{\ in \ }L_0^{p^\prime}(\omega \times \Pi)\,.
\end{equation}
\end{lemma}
\begin{proof}
Reasoning as in Lemma \ref{pseudo_lemma1}, the estimate (\ref{esti_gorro}) implies the existence $\hat{P}:\omega \times \Pi \to \mathbb{R}$ such that (\ref{convPgorro1}) holds. By semicontinuity and the previous estimate of $\hat{P}_{\varepsilon}$, we have $$\int_{\omega \times \Pi}\left\vert \hat{P} \right\vert^{p^\prime} dx^{\prime}dy\leq C,$$ which shows that $\hat{P}$ belongs to $L^{p^\prime}(\omega \times \Pi)$.
\end{proof}
Next, we give a convergence result for $\hat{u}_{\varepsilon}$.
\begin{lemma}\label{lemma_gorro} Assume that $9/5\leq p<+\infty$. 
For a subsequence of $\varepsilon$ still denote by $\varepsilon$, \begin{itemize}
\item[i)] if $\eta_{\varepsilon}\approx \varepsilon$ with $\eta_\varepsilon/\varepsilon\to \lambda$, $0<\lambda<+\infty$, then there exist $\hat{u}\in L^p(\omega;W^{1,p}_{\sharp}(Y)^3)$, with $\int_{Y}\hat u_3\,dy=0$ and $\hat u=0$ on $y_3=\{0,h(y')\}$, such that
\begin{equation}\label{convUgorro1}
\eta_\varepsilon^{-\frac{p}{p-1}}\hat{u}_{\varepsilon}\rightharpoonup \hat{u}\text{\ in \ }L^p(\omega;W^{1,p}(Y)^3),
\end{equation}
\begin{equation}\label{div0}
{\rm div}_{\lambda}\hat{u}=0 \text{\ in \ }\omega\times Y,
\end{equation}
where ${\rm div}_{\lambda}=\lambda {\rm div}_{y^\prime}+ \partial_{y_3}$,
\item[ii)] if $\eta_{\varepsilon}\ll \varepsilon$, then there exist $\hat{u}\in L^p(\omega;W^{1,p}_{\sharp}(Y)^3)
$,  with $\hat u=0$ on $y_3=\{0,h(y')\}$, $\int_{Y}\hat u_3\,dy=0$ and $\hat u_3$ independent of $y_3$, such that
\begin{equation}\label{converneUgorroUtilde}
\eta_\varepsilon^{-\frac{p}{p-1}}\hat{u}_{\varepsilon}\rightharpoonup \hat u\text{\ in \ }L^p(\omega;W^{1,p}(Y)^3),
\end{equation}
\begin{equation}\label{div0-2}
{\rm div}_{y'}\hat{u}^\prime=0 \text{\ in \ }\omega\times Y,
\end{equation}
\item[iii)]  if $\eta_{\varepsilon}\gg \varepsilon$, then we have that
\begin{equation*}\label{convUgorro2}
\eta_{\varepsilon}^{-\frac{p}{p-1}}\hat{u}_{\varepsilon}\rightharpoonup (\tilde v',0)\text{\ in \ } W^{1,p}(0,h_{\rm min};L^p(\omega)^3)\,.\end{equation*}
where $(\tilde v',0)$ is the weak limit in $W^{1,p}(0,h_{\rm min};L^p(\omega)^3)$ of $\tilde v_\varepsilon$ given in Lemma \ref{LemmaConvergenceUtilde}.
\end{itemize}
Moreover, in the cases $\eta_\varepsilon\approx \varepsilon$ and $\eta_\varepsilon\ll \varepsilon$, we have
\begin{equation}\label{divmacro1}
{\rm div}_{x^{\prime}}\left( \int_Y \hat{u}^{\prime}(x^{\prime},y)dy \right)=0 \text{\ in \ }\omega,\quad \left( \int_Y \hat{u}^{\prime}(x^{\prime},y)dy \right)\cdot n=0\text{\ on \ }\partial \omega\,.
\end{equation}

\end{lemma}
\begin{proof}
We proceed in four steps. \\
{\it Step 1.} Case $\eta_{\varepsilon}\approx \varepsilon$. In this case, the estimates (\ref{estim_velocidad_gorro1SIMPLE})-(\ref{estim_velocidad_gorro2}) read 
\begin{equation}\label{estimate_case1}
\left\Vert \hat{u}_{\varepsilon}\right\Vert_{L^p(\omega\times Y)^3}\leq C \eta_\varepsilon^{\frac{p}{p-1}}, \quad \left\Vert D_{y}\hat{u}_{\varepsilon}\right\Vert_{L^p(\omega\times Y)^{{3\times3}}}\leq C \eta_\varepsilon^{\frac{p}{p-1}}.
\end{equation}
The above estimates imply the existence $\hat{u}: \omega \times Y \to \mathbb{R}^3$, such that, up to a subsequence, convergences (\ref{convUgorro1}) holds. By semicontinuity and the estimates given in (\ref{estimate_case1}), we have $$\int_{\omega \times Y}\left\vert \hat{u} \right\vert^pdx^{\prime}dy\leq C,\quad \int_{\omega \times Y}\left\vert D_{y}\hat{u} \right\vert^pdx^{\prime}dy\leq C,$$ which shows that $\hat{u}\in L^p(\omega;W^{1,p}(Y)^3)$.

It would remain to prove the $Y^{\prime}$-periodicity of $\hat{u}$ in $y^{\prime}$. This can be obtain by proceeding as in Lemma 5.4 in \cite{grau1}.

Since ${\rm div}_{\eta_\varepsilon}\tilde{u}_{\varepsilon}=0$ in $\Omega$, then by definition of $\hat{u}_{\varepsilon}$ we have $\varepsilon^{-1}{\rm div}_{y^\prime}\hat{u}^\prime_\varepsilon+\eta_\varepsilon^{-1}\partial_{y_3}\hat{u}_{\varepsilon,3}=0$. Multiplying by $\eta_\varepsilon^{-1/(p-1)}$ we obtain
\begin{equation*}\label{div_important} 
{\eta_\varepsilon\over \varepsilon}\eta_\varepsilon^{-\frac{p}{p-1}}{\rm div}_{y^\prime}\hat{u}_{\varepsilon}^\prime + \eta_\varepsilon^{-\frac{p}{p-1}}\partial_{y_3}\hat{u}_{\varepsilon,3}=0,\quad \text{\ in\ } \omega \times Y,
\end{equation*}
which, combined with (\ref{convUgorro1}) and $\eta_\varepsilon/\varepsilon\to \lambda$, proves (\ref{div0}).\\

{\it Step 2.} Case $\eta_{\varepsilon}\ll \varepsilon$. In this case, from the second estimates (\ref{estim_velocidad_gorro1SIMPLE}) and 
(\ref{estim_velocidad_gorro2}), up to a subsequence and using a semicontinuity argument, there exists $\hat{u}\in W^{1,p}(0,h(y'); L^p(\omega\times Y^{\prime})^3)$ such that 
\begin{equation}\label{convernecesary_2}
\eta_{\varepsilon}^{-\frac{p}{p-1}}\hat{u}_{\varepsilon}\rightharpoonup \hat{u}\text{\ in \ }W^{1,p}(0,h(y');L^{p}(\omega\times Y^{\prime})^3),
\end{equation}
which implies
\begin{equation*}\label{conv_div_123_case_2}
\eta_\varepsilon^{-{p\over p-1}}{\rm div}_{y^\prime}\hat{u}_{\varepsilon}^{\prime}\rightharpoonup {\rm div}_{y^\prime}\hat{u}^{\prime}\text{\ in \ }W^{1,p}(0,h(y');W^{-1,p'}(Y';L^p(\omega))).
\end{equation*}

Since ${\rm div}_{\eta_\varepsilon}\tilde{u}_{\varepsilon}=0$ in $\Omega$, then by definition of $\hat{u}_{\varepsilon}$ we have $\varepsilon^{-1}{\rm div}_{y^\prime}\hat{u}^\prime_\varepsilon+\eta_\varepsilon^{-1}\partial_{y_3}\hat{u}_{\varepsilon,3}=0$. Multiplying by $\eta_\varepsilon^{-1/(p-1)}$, we obtain
\begin{equation}\label{div_important_2} 
{\eta_\varepsilon\over \varepsilon}\eta_\varepsilon^{-\frac{p}{p-1}}{\rm div}_{y^\prime}\hat{u}_{\varepsilon}^\prime + \eta_\varepsilon^{-\frac{p}{p-1}}\partial_{y_3}\hat{u}_{\varepsilon,3}=0,\quad \text{\ in\ } \omega \times Y.
\end{equation}
From the above convergences and $\eta_\varepsilon/\varepsilon\to 0$, we can deduce that $\partial_{y_3}\hat u_3=0$, and so $\hat u_3$ does not depend on $y_3$.

Now, we prove (\ref{div0-2}). To do this, we consider $\varphi\in C^1_c(\omega\times Y')$ as test function in (\ref{div_important_2}), which gives
$$\int_{\omega\times Y}{\rm div}_{y'}\hat u_\varepsilon'\,\varphi(x',y')\,dx'dy=0.$$
Multiplying by $\eta_\varepsilon^{-{p\over p-1}}$ and from convergence (\ref{convernecesary_2}), we get (\ref{div0-2}).

In order to proof the $Y^{\prime}$-periodicity of $\hat{u}$ in $y^{\prime}$, we proceed similarly to the step 1.\\

{\it Step 3.} In order to prove (\ref{divmacro1}), let us first prove the following relation between $\tilde{v}$ and $\hat{u}$ for the cases $\eta_\varepsilon \approx \varepsilon$ or $\eta_\varepsilon \ll \varepsilon$,
\begin{equation}\label{relation}
{1\over |Y'|}\int_{Y}\hat{u}(x^{\prime},y)dy=\int_0^{h_{\rm max}}\tilde{v}(x^{\prime},y_3)dy_3.
\end{equation}
For this, let us consider $\varphi\in C_c^1(\omega)$. We observe that using the definition (\ref{uhat}) of $\hat{u}_{\varepsilon}$, we obtain
\begin{equation*}
\begin{array}{l}\displaystyle
\eta_{\varepsilon}^{-\frac{p}{p-1}}\int_{\omega}\int_{Y}\hat{u}_{\varepsilon}(x^{\prime},y)\varphi(x^{\prime})dydx^{\prime}\\
\\
\displaystyle
=\eta_{\varepsilon}^{-\frac{p}{p-1}}\displaystyle\sum_{k^{\prime}\in T_{\varepsilon}}\int_{Y^{\prime}_{k^{\prime},{\varepsilon}}}\int_{Y}\tilde{u}_{\varepsilon}({\varepsilon}k^{\prime}+{\varepsilon}y^{\prime},y_3)\,\varphi({\varepsilon}k^{\prime}+{\varepsilon}y^{\prime})dydx^{\prime}+O_\varepsilon.
\end{array}
\end{equation*}
We observe that $\tilde{u}_{\varepsilon}$ and $\varphi$ do not depend on $x^{\prime}$, then we can deduce
\begin{equation*}
\begin{array}{l}\displaystyle
\eta_{\varepsilon}^{-\frac{p}{p-1}}\int_{\omega}\int_{Y}\hat{u}_{\varepsilon}(x^{\prime},y)\varphi(x^{\prime})dydx^{\prime}\\
\displaystyle
=\eta_{\varepsilon}^{-\frac{p}{p-1}}\varepsilon^2|Y'|\displaystyle\sum_{k^{\prime}\in T_{\varepsilon}}\int_{Y^{\prime}}\int_0^{h(y')}\tilde{u}_{\varepsilon}({\varepsilon}k^{\prime}+{\varepsilon}y^{\prime},y_3)\,\varphi({\varepsilon}k^{\prime}+{\varepsilon}y^{\prime})dy_3dy^{\prime}+O_\varepsilon.
\end{array}
\end{equation*}
By the change of variables (\ref{CV}) and the $Y'$-periodicity of $h$, we obtain
$$\begin{array}{l}\displaystyle
\eta_{\varepsilon}^{-\frac{p}{p-1}}\int_{\omega}\int_{Y}\hat{u}_{\varepsilon}(x^{\prime},y)\varphi(x^{\prime})dydx^{\prime}\\
\displaystyle\qquad=\eta_{\varepsilon}^{-\frac{p}{p-1}}|Y'|\displaystyle\sum_{k^{\prime}\in T_{\varepsilon}}\int_{Y^{\prime}_{k^{\prime},{\varepsilon}}}\int_0^{h(x'/\varepsilon)}\tilde{u}_{\varepsilon}(x^{\prime},y_3)\,\varphi(x^{\prime})dy_3dx^{\prime}+O_\varepsilon\\
\displaystyle\qquad=\eta_{\varepsilon}^{-\frac{p}{p-1}}|Y'|\displaystyle\sum_{k^{\prime}\in T_{\varepsilon}}\int_{Y^{\prime}_{k^{\prime},{\varepsilon}}}\int_0^{h_{\rm max}}\tilde{v}_{\varepsilon}(x^{\prime},y_3)\,\varphi(x^{\prime})dy_3dx^{\prime}+O_\varepsilon\\
\displaystyle\qquad=\eta_{\varepsilon}^{-\frac{p}{p-1}}|Y'|\displaystyle\int_{\Omega}\tilde{v}_{\varepsilon}(x^{\prime},y_3)\,\varphi(x^{\prime})dy_3dx^{\prime}+O_\varepsilon.
\end{array}$$
Taking into account the convergences (\ref{convtilde1}), (\ref{convUgorro1}) and  (\ref{converneUgorroUtilde}), we obtain (\ref{relation}) for  the cases $\eta_\varepsilon\approx \varepsilon$ or $\eta_\varepsilon\ll \varepsilon$. Since $\tilde v_3=0$, we deduce that $\int_{Y}\hat u_3\,dy=0$ a.e. in $\omega$. Finally, relation (\ref{relation}) together with (\ref{divmacro_tilde1}) implies (\ref{divmacro1}).
\par

{\it Step 4.} Case $\eta_{\varepsilon}\gg \varepsilon$. In this case, by the second estimate in (\ref{estim_velocidad_gorro1SIMPLE}) and estimate (\ref{estim_velocidad_gorro2}), up to a subsequence and using a semicontinuity argument, there exists $\hat{u}\in W^{1,p}(0,h(y'); L^p(\omega\times Y^{\prime})^3)$ such that 
\begin{equation}\label{convernecesary_case_2}
\eta_{\varepsilon}^{-\frac{p}{p-1}}\hat{u}_{\varepsilon}\rightharpoonup \hat{u}\text{\ in \ }W^{1,p}(0,h(y');L^{p}(\omega\times Y^{\prime})^3).
\end{equation}

Since $\varepsilon^{-1}\eta_\varepsilon^{-\frac{1}{p-1}}D_{y'}\hat{u}_{\varepsilon}$ is bounded in $L^p(\omega \times Y)^3$, we observe that $\eta_\varepsilon^{-\frac{p}{p-1}}D_{y'}\hat{u}_{\varepsilon}$  is also bounded, and tends to zero. This together with (\ref{convernecesary_case_2}) implies
\begin{equation*}
\eta_{\varepsilon}^{-\frac{p}{p-1}}D_{y'}\hat{u}_{\varepsilon}\rightharpoonup 0\text{\ in \ }W^{1,p}\left(0,h(y');L^p(\omega\times Y^{\prime})^{3\times 2}\right),
\end{equation*}
 and so $\hat{u}$ does not depend on $y'$.

Taking into account (\ref{convtilde1_super_2}), proceeding as in (\ref{relation}) but in $\omega\times \Pi^-$, we obtain
\begin{equation*}\label{relation2}
{1\over |Y'|}\int_{\Pi^-}\hat{u}(x^{\prime},y)dy=\int_0^{h_{\rm min}} \tilde v(x',y_3)\,dy_3\,.
\end{equation*}
Since $\hat{u}$ does not depend on $y^\prime$, we have that $\hat{u}=(\tilde{v}^\prime,0)$.
\end{proof}

\section{Effective models}\label{S6}
In this section, we will multiply system (\ref{2}) by a test function having the form of the limit $\hat{u}$ (as explicated in Lemma \ref{lemma_gorro}), and we will use the convergences given in the previous section in order to identify the effective model in every case. 

\begin{theorem}\label{thm_homenized}
Assume that $9/5\leq p< +\infty$. We distingue three cases:
\begin{itemize}
\item[i)] If $\eta_{\varepsilon}\approx \varepsilon$, with $\eta_\varepsilon/\varepsilon\to \lambda$, $0<\lambda<+\infty$, then $(\eta_\varepsilon^{-\frac{p}{p-1}}\hat{u}_{\varepsilon},\hat{P}_\varepsilon)$ converges to the unique solution $(\hat{u}(x^{\prime},y),\tilde{P}(x^{\prime}))$ in $L^p(\omega; W^{1,p}(Y)^3)\times L^{p'}_0(\omega)\cap W^{1,p'}(\omega)$, with $\int_{Y}\hat u_3\,dy=0$, of the effective problem
\begin{equation}
\left\{
\begin{array}{rcl}
\displaystyle -\mu\,{\rm div}_{\lambda}\left( S\left(\mathbb{D}_{\lambda}\left[\hat{u} \right] \right) \right) + \nabla_{\lambda} \hat{q} 
&=& 
 f^{\prime}-\nabla_{x^{\prime}}\tilde{P}\quad \text{\ in \ }\omega \times Y,\\
 \displaystyle{\rm div}_{\lambda} \hat{u} & =&  0\quad \text{\ in \ }\omega \times Y,\\
 \hat{u}&=&0\text{\ on \ }y_3=0,h(y'),\\
 \displaystyle{\rm div}_{x^{\prime}}\left( \int_Y \hat{u}^{\prime}(x^{\prime},y)dy \right)&=&0 \text{\ in \ }\omega,\\
 \displaystyle \left( \int_Y \hat{u}^{\prime}(x^{\prime},y)dy \right)\cdot n& =& 0\text{\ on \ }\partial \omega,\\
  y'\to \hat{u},\hat{q}\quad Y'-\text{periodic},
\end{array}
\right. \label{effective1}%
\end{equation}
where $\mathbb{D}_{\lambda}\left[\cdot\right]=\lambda\mathbb{D}_{y^\prime}\left[\cdot\right]+ \partial_{y_3}\left[\cdot\right]$, $\nabla_{\lambda}=(\lambda \nabla_{y^\prime},\partial_{y_3})$ and ${\rm div}_\lambda=\lambda{\rm div}_{y^\prime} + \partial_{y_3}$.
\item[ii)] If $\eta_{\varepsilon}\ll \varepsilon$, then $(\eta_\varepsilon^{-\frac{p}{p-1}}\hat{u}_\varepsilon,\hat{P}_\varepsilon)$ converges to the unique solution $(\hat{u}(x^{\prime},y),\tilde{P}(x^{\prime}))$ in $L^p(\omega;W^{1,p}(Y)^3)\times L^{p'}_0(\omega)\cap W^{1,p'}(\omega)$, with $\int_{Y} \hat{u}_3 dy=0$ and $\hat u_3$ independent of $y_3$, of the effective problem
\begin{equation}
\left\{
\begin{array}{rcl}
\displaystyle -\mu\,\partial_{y_3}\left(S\left(\partial_{y_3}\left[\hat{u}^\prime \right]\right) \right) + \nabla_{y^\prime} \hat{q} 
&=& 
 f^{\prime}-\nabla_{x^{\prime}}\tilde{P}\quad \text{\ in \ }\omega \times Y\\
\displaystyle{\rm div}_{y^\prime} \hat{u}^\prime & =&  0\quad \text{\ in \ }\omega \times Y,\\
\displaystyle
 \hat{u}^\prime&=&0\text{\ on \ }y_3=0,h(y'),\\
 \displaystyle{\rm div}_{x^{\prime}}\left( \int_{Y} \hat{u}^{\prime}(x^{\prime},y)dy \right)&=&0 \text{\ in \ }\omega,\\
 \displaystyle \left( \int_{Y} \hat{u}^{\prime}(x^{\prime},y)dy \right)\cdot n& =& 0\text{\ on \ }\partial \omega,\\
  y'\to \hat{u}^\prime,\hat{q}\quad Y^{\prime}-\text{periodic}.
\end{array}
\right. \label{effective2}%
\end{equation}
\item[iii)]  If $\eta_\varepsilon \gg \varepsilon$, then the extension $(\eta_\varepsilon^{-{p\over p-1}}\tilde{v}_\varepsilon,\tilde{P}_\varepsilon)$ converges to the unique solution $(\tilde{v}(x^{\prime},y_3),\tilde{P}(x^{\prime}))$ in $W^{1,p}(0,h_{\rm min};L^p(\omega)^3)\times L^{p'}_0(\omega)\cap W^{1,p'}(\omega)$, with $\tilde v_3=0$, of the effective problem
\begin{equation}\label{Homogenized_3}
\left\{\begin{array}{rcl}\displaystyle
\displaystyle-\mu\,\partial_{y_3}S\left(\partial_{y_3}\tilde{v}^\prime\right)&=&2^{p\over 2}\left(f^\prime(x^\prime)-\nabla_{x^\prime}\tilde{P}(x^\prime)\right)\ \text{ in\ }\Omega^-\,,
\\
\displaystyle\tilde{v}^\prime&=&0 \quad\text{ on\ }y_3=0,h_{\rm min},
\\
\displaystyle {\rm div}_{x^\prime}\left(\int_0^{h_{\rm min}} \tilde{v}^\prime(x^\prime,y_3)dy_3\right)&=&0 \quad\text{ in\ }\omega\,,\\
\displaystyle \left(\int_0^{h_{\rm min}} \tilde{v}^\prime(x^\prime,y_3)dy_3\right)\cdot n&=&0\quad\text{ in\ }\omega\,.\end{array}\right.
\end{equation}
\end{itemize}
\end{theorem}
\begin{proof}
First of all, we choose a test function $\varphi(x^{\prime},y)\in \mathcal{D}(\omega;C_{\sharp}^{\infty}(Y)^3)$. Multiplying (\ref{2}) by $\varphi(x^\prime,x^\prime/\varepsilon,y_3)$, integrating by parts, and taking into account that reasoning as in the proof of Lemma \ref{lem_estim_nabla_pressure} we get
$$\int_{\widetilde\Omega_\varepsilon}(\tilde u_\varepsilon\cdot \nabla_{\eta_\varepsilon})\tilde u_\varepsilon\, \varphi\,dx'dy_3=O_\varepsilon,$$
then we have
\begin{equation*}\begin{array}{l}\displaystyle
\mu\int_{\widetilde \Omega_\varepsilon}S\left(\mathbb{D}_{\eta_\varepsilon} \left[\tilde{u}_\varepsilon \right] \right):\left(\mathbb{D}_{x^\prime}\left[\varphi\right]+\frac{1}{\varepsilon} \mathbb{D}_{y^\prime}\left[\varphi\right]+\frac{1}{\eta_\varepsilon}\partial_{y_3}\left[\varphi\right]\right) dx^{\prime}dy_3\\
\displaystyle +\int_{\widetilde \Omega_\varepsilon} \nabla_{\eta_\varepsilon}\tilde{p}_\varepsilon\, \varphi\,dx^{\prime}dy_3=\int_{\widetilde \Omega_\varepsilon}f'\cdot \varphi'\,dx^\prime dy_3+ O_\varepsilon\,.
\end{array}
\end{equation*}
Taking into account the prolongation  of the pressure, we have
$$\int_{\widetilde \Omega_\varepsilon} \nabla_{\eta_\varepsilon}\tilde{p}_\varepsilon\, \varphi^\prime\,dx^{\prime}dy_3=\int_{\Omega} \nabla_{\eta_\varepsilon}\tilde{P}_\varepsilon\,\varphi\,dx^{\prime}dy_3,$$
and so
\begin{equation}\begin{array}{l}\displaystyle
\mu\int_{\widetilde \Omega_\varepsilon}S\left(\mathbb{D}_{\eta_\varepsilon} \left[\tilde{u}_\varepsilon \right] \right):\left(\mathbb{D}_{x^\prime}\left[\varphi\right]+\frac{1}{\varepsilon} \mathbb{D}_{y^\prime}\left[\varphi\right]+\frac{1}{\eta_\varepsilon}\partial_{y_3}\left[\varphi\right]\right) dx^{\prime}dy_3\\
\displaystyle -\int_{\Omega} \tilde{P}_\varepsilon\, {\rm div}_{x^\prime} \varphi^\prime\,dx^{\prime}dy_3-\frac{1}{\varepsilon}\int_{\Omega} \tilde{P}_\varepsilon\, {\rm div}_{y^\prime} \varphi^\prime\,dx^{\prime}dy_3-\frac{1}{\eta_\varepsilon}\int_{\Omega}\tilde{P}_\varepsilon\,\partial_{y_3}\varphi_3\,dx^{\prime}dy_3\\
\displaystyle=\int_{\widetilde \Omega_\varepsilon}f'\cdot \varphi'\,dx^\prime dy_3+ O_\varepsilon\,.
\end{array}\label{formulaciondebiltilde}
\end{equation}
By the change of variables given in Remark \ref{remarkCV}, we obtain
\begin{equation*}\label{formvarcvuprime}\begin{array}{l}\displaystyle
\mu\int_{\omega\times Y}S\left(\frac{1}{\varepsilon}\mathbb{D}_{y^\prime} \left[\hat{u}_\varepsilon \right] +\frac{1}{\eta_\varepsilon}\partial_{y_3}\left[\hat{u}_\varepsilon \right]\right):\left(\frac{1}{\varepsilon} \mathbb{D}_{y^\prime}\left[\varphi\right]+\frac{1}{\eta_\varepsilon}\partial_{y_3}\left[\varphi\right]\right)dx^{\prime}dy \\
\displaystyle
-\int_{\omega\times \Pi} \hat{P}_\varepsilon\, {\rm div}_{x^\prime} \varphi^\prime\,dx^{\prime}dy-\frac{1}{\varepsilon}\int_{\omega\times \Pi} \hat{P}_\varepsilon\, {\rm div}_{y^\prime} \varphi^\prime\,dx^{\prime}dy \\
\displaystyle
-\frac{1}{\eta_\varepsilon}\int_{\omega\times \Pi}\hat{P}_\varepsilon\,\partial_{y_3}\varphi_3\,dx^{\prime}dy=\int_{\omega\times Y}f'\cdot \varphi'\,dx^\prime dy+O_\varepsilon\,,
\end{array}
\end{equation*}
which can be written by
\begin{equation}\label{formvarcv1-0}\begin{array}{l}\displaystyle
\mu\,\int_{\omega\times Y}S\left({\eta_\varepsilon\over \varepsilon}\eta_\varepsilon^{-\frac{p}{p-1}}\mathbb{D}_{y^\prime} \left[\hat{u}_\varepsilon \right] +\eta_\varepsilon^{-\frac{p}{p-1}}\partial_{y_3}\left[\hat{u}_\varepsilon \right]\right):\left( \frac{\eta_\varepsilon}{\varepsilon}\mathbb{D}_{y^\prime}\left[\varphi\right]+\partial_{y_3}\left[\varphi\right]\right)dx^{\prime}dy\\
\displaystyle-\int_{\omega\times \Pi} \hat{P}_\varepsilon\, {\rm div}_{x^\prime} \varphi^\prime\,dx^{\prime}dy-\frac{1}{\varepsilon}\int_{\omega\times \Pi} \hat{P}_\varepsilon\,{\rm div}_{y^\prime}\varphi^\prime\,dx'dy\\
\displaystyle
-\frac{1}{\eta_\varepsilon}\int_{\omega\times \Pi} \hat{P}_\varepsilon\,\partial_{y_3}\varphi_3\,dx'dy=\int_{\omega\times Y}f'\cdot \varphi'\,dx^\prime dy+O_\varepsilon\,.
\end{array}
\end{equation}

This variational formulation will be useful in the following steps.

We proceed in three steps. \\
{\it Step 1.} Case $\eta_{\varepsilon}\approx \varepsilon$, with $\eta_\varepsilon/\varepsilon\to \lambda$, $0<\lambda<+\infty$.  

First, we prove that $\hat{P}$ does not depend on the microscopic variable $y$. To do this, we consider as test function $\eta_\varepsilon \varphi(x^\prime,x^\prime/\varepsilon,y_3)$ in (\ref{formvarcv1-0}), taking into account the estimates in (\ref{estim_velocidad_gorro1}) and passing to the limit when $\varepsilon$ tends to zero by using convergence (\ref{convPgorro1}), we have 
$$\int_{\omega\times \Pi} \hat{P}\,{\rm div}_{\lambda}\varphi\,dx'dy =0,$$
which shows that $\hat{P}$ does not depend on $y$.\\

For all $\varphi \in \mathcal{D}(\omega;C_{\sharp}^{\infty}(Y)^3)$ with ${\rm div}_{\lambda}\varphi=0$ in $\omega \times Y$ and ${\rm div}_{x^\prime}(\int_Y \varphi^\prime\,dy)=0$ in $\omega$, we choose $\phi_\varepsilon=(\phi^{\prime}_{\varepsilon}, \phi_{\varepsilon,3})$  defined by $$\phi^\prime_\varepsilon=\lambda{\varepsilon\over \eta_\varepsilon}\varphi^\prime-\eta_\varepsilon^{-\frac{p}{p-1}}\hat{u}_\varepsilon^\prime,\quad \phi_{\varepsilon,3}=\varphi_3-\eta_\varepsilon^{-{p\over p-1}}\hat u_{\varepsilon,3}\,,$$ as a test function in (\ref{formvarcv1-0}). Due to monotonicity, we have
\begin{equation*}\begin{array}{l}\displaystyle
\mu \int_{\omega\times Y}S\left(\frac{\eta_\varepsilon}{\varepsilon}\mathbb{D}_{y^\prime} \left[\varphi \right] +\partial_{y_3}\left[\varphi \right]\right):\left( \frac{\eta_\varepsilon}{\varepsilon}\mathbb{D}_{y^\prime}\left[\phi_\varepsilon \right]+\partial_{y_3}\left[\phi_\varepsilon \right]\right)dx^{\prime}dy \\
\displaystyle
-\int_{\omega\times \Pi} \hat{P}_\varepsilon\, {\rm div}_{x^\prime} \phi^\prime_{\varepsilon}\,dx^{\prime}dy
\ge\int_{\omega\times Y}f'\cdot \phi^\prime_{\varepsilon}\,dx^\prime dy+O_\varepsilon\,.
\end{array}
\end{equation*}
Thus, we can use  the convergences (\ref{convPgorro1}) and (\ref{convUgorro1}). If we argue similarly as in \cite{bourgeat_fissure}, we have that the convergence of the pressure is in fact strong. This implies that the convergence of the pressure $\hat P_\varepsilon$ is also in fact strong (see Proposition 2.9 in \cite{Ciora2}).  Then, when passing to the limit, the second term contributes nothing because the limit of $\hat{P}_\varepsilon$ does not depend on $y$ and $\hat u'$ satisfies (\ref{divmacro1}). Taking into account that $\lambda\,\varepsilon/\eta_\varepsilon \to 1$, we obtain
\begin{equation*}\begin{array}{l}\displaystyle
\mu \int_{\omega\times Y}S\left(\lambda\mathbb{D}_{y^\prime} \left[\varphi \right] +\partial_{y_3}\left[\varphi \right]\right):\left(\lambda \mathbb{D}_{y^\prime}\left[\varphi-\hat{u} \right]+\partial_{y_3}\left[\varphi-\hat{u} \right]\right)dx^{\prime}dy \\
\displaystyle
\ge\int_{\omega\times Y}f'\cdot (\varphi^\prime-\hat{u}^\prime)\,dx^\prime dy\,,
\end{array}
\end{equation*}
which, due to Minty Lemma \cite{Lions2}, is equivalent to
\begin{equation*}
-\mu\,{\rm div}_{\lambda}\left( S\left(\mathbb{D}_{\lambda}\left[\hat{u}' \right] \right) \right)=f' \text{\ \ in \ \ }\omega \times Y.
\end{equation*}
By density 
\begin{equation}\label{formvarcv1_lim}
\mu\,\int_{\omega \times Y}S\left(\mathbb{D}_{\lambda}\left[\hat{u} \right] \right): \mathbb{D}_{\lambda}\left[\varphi \right]dx^\prime dy=\int_{\omega \times Y}f^\prime\,\varphi^\prime\,dx^\prime dy
\end{equation}
holds for every function $\varphi$ in the Hilbert space $V$ defined by 
$$V=\left\{\begin{array}{l}
\displaystyle
\varphi(x^\prime,y)\in L^p(\omega;W^{1,p}_{\sharp}(Y)^3),\ \text{ such that }\\
\displaystyle{\rm div}_{\lambda}\varphi(x^\prime,y)=0\ \text{ in }\omega\times Y,\quad {\rm div}_{x^\prime}\left(\int_{Y}\varphi(x^\prime,y)\,dy\right)=0\ \text{ in }\omega,\\
\displaystyle 
\varphi(x^\prime,y)=0\ \text{ in }\omega\times Y_s,\quad \left(\int_{Y}\varphi(x^\prime,y)\,dy\right)\cdot n=0\ \text{ on }\omega
\end{array}\right\}\,.$$
By Lax-Milgram lemma, the variational formulation (\ref{formvarcv1_lim}) in the Hilbert space $V$ admits a unique solution $\hat{u}$ in $V$. Reasoning as in \cite{Allaire0}, the orthogonal of $V$ with respect to the usual scalar product in $L^p(\omega\times Y)$ is made of gradients of the form $\nabla_{x^\prime}q(x^\prime)+\nabla_{\lambda}\hat{q}(x',y)$, with $q(x^\prime)\in L^{p^\prime}_0(\omega)$ and $\hat q(x^\prime,y)\in L^{p^\prime}(\omega;W^{1,p}_{\sharp}(Y))$.  Therefore, by integration by parts, the variational formulation (\ref{formvarcv1_lim}) is equivalent to the effective system (\ref{effective1}).  It remains to prove that the pressure $\tilde P (x^\prime)$, arising as a Lagrange multiplier of the incompressibility constraint 
${\rm div}_{x^\prime}(\int_{Y}\hat u(x^\prime,y)dy)=0$, is the same as the limit of the pressure $\tilde{P}_\varepsilon$. This can be easily done by multiplying equation (\ref{2}) by a test function with ${\rm div}_{\lambda}$ equal to zero, and identifying limits. Since (\ref{effective1}) admits a unique solution, then the complete sequence $(\eta_\varepsilon^{-p/(p-1)}\hat{u}_{\varepsilon},\hat{P}_\varepsilon)$ converges to the solution $(\hat u(x^\prime,y),\hat P(x^\prime))$. Finally, from Theorem 8 in \cite{BourMik} we have that system  (\ref{effective1}) has a unique solution and  moreover $\tilde P\in W^{1,p'}(\omega)$. \\

{\it Step 2.}  Case $\eta_{\varepsilon}\ll \varepsilon$. 

First, we prove that $\hat{P}$ does not depend on the vertical variable $y_3$. To do this, we consider as test function $(0, \eta_\varepsilon \varphi_3(x^\prime,x^\prime/\varepsilon,y_3))$ in (\ref{formvarcv1-0}), taking into account the estimates in (\ref{estim_velocidad_gorro1}) and passing to the limit when $\varepsilon$ tends to zero by using the convergence (\ref{convPgorro1}), we have 
$$\int_{\omega\times \Pi} \hat{P}\,\partial_{y_3}\varphi_3\,dx^\prime dy=0,$$
which shows that $\hat{P}$ does not depend on $y_3$.\\

Let us now prove that $\hat{P}$ does not depend on the microscopic variable $y^\prime$. For this, we take now as test function  $(\varepsilon \varphi^\prime(x^\prime,x^\prime/ \varepsilon,y_3),0)$  in (\ref{formvarcv1-0}).   By using estimates in (\ref{estim_velocidad_gorro1}) and the convergence (\ref{convPgorro1}), we get
$$\int_{\omega\times \Pi}\hat{P}\,{\rm div}_{y'}\varphi^\prime\,dx^\prime dy=0$$
which implies that $\hat{P}$ does not depend on $y^\prime$.  Thus, we conclude that $\hat{P}$ does not depend on the entire variable $y$.
\\

For all $\varphi \in \mathcal{D}(\omega;C_{\sharp}^{\infty}(Y)^3)$ with $\varphi_3$ independent of $y_3$, ${\rm div}_{y'}\varphi'=0$ in $\omega \times Y$ and ${\rm div}_{x^\prime}(\int_Y \varphi^\prime\,dy)=0$ in $\omega$, we choose $\phi_\varepsilon=\varphi-\eta_\varepsilon^{-\frac{p}{p-1}}\hat{u}_\varepsilon$, as a test function in (\ref{formvarcv1-0}). Using monotonicity, we have
\begin{equation*}\begin{array}{l}\displaystyle
\mu \int_{\omega\times Y}S\left({\eta_\varepsilon \over \varepsilon}\mathbb{D}_{y^\prime} \left[\varphi \right] +\partial_{y_3}\left[\varphi \right]\right):\left( {\eta_\varepsilon \over \varepsilon}\mathbb{D}_{y^\prime}\left[\phi_\varepsilon \right]+\partial_{y_3}\left[\phi_\varepsilon \right]\right)dx^{\prime}dy \\
\displaystyle
-\int_{\omega\times \Pi} \hat{P}_\varepsilon\, {\rm div}_{x^\prime} \phi^\prime_{\varepsilon}\,dx^{\prime}dy
\ge\int_{\omega\times Y}f'\cdot \phi^\prime_{\varepsilon}\,dx^\prime dy+O_\varepsilon\,.
\end{array}
\end{equation*}
Thus, we can use  the convergences (\ref{convPgorro1}) and (\ref{converneUgorroUtilde}). If we argue similarly as the step 1, we have that the convergence of the pressure $\hat{P}_\varepsilon$ is strong. Then, when passing to the limit, the second term contributes nothing because the limit of $\hat{P}_\varepsilon$ does not depend on $y$ and $\hat{u}'$ satisfies (\ref{divmacro1}). We obtain
\begin{equation*}\begin{array}{l}\displaystyle
\mu \int_{\omega\times Y}S\left(\partial_{y_3}\left[\varphi^\prime \right]\right):\partial_{y_3}\left[\varphi^\prime-\hat{u}^\prime \right]dx^{\prime}dy \\
\displaystyle
\ge\int_{\omega\times Y}f'\cdot (\varphi^\prime-\hat{u}^\prime)\,dx^\prime dy+O_\varepsilon\,,
\end{array}
\end{equation*}
which, due to Minty Lemma \cite{Lions2}, is equivalent to
\begin{equation*}
-\mu\,\partial_{y_3}\left( S\left(\partial_{y_3}\left[\hat u^\prime \right] \right) \right)=f^\prime \text{\ \ in \ \ }\omega \times Y.
\end{equation*}
By density, and reasoning as in Step 1, this problem is equivalent to the effective system (\ref{effective2}). Observe that the condition (\ref{div0-2}) implies that $\hat{q}$ does not depend on $y_3$. Finally, from Theorem 8 in \cite{BourMik} we have that system  (\ref{effective2}) has a unique solution and  moreover $\tilde P\in W^{1,p'}(\omega)$.\\

{\it Step 3.} Case $\eta_\varepsilon \gg \varepsilon$.  From Lemma \ref{LemmaConvergenceUtilde} and Lemma \ref{lemma_gorro}, we take into account that we are going to obtain an effective problem for the pressure in $\Omega^-$ without involving the microstructure of the domain $\widetilde \Omega_\varepsilon$. Thus, we choose in (\ref{formulaciondebiltilde}) the following test function  
$\varphi_{\varepsilon}(x^\prime,y_3)=(\varphi^\prime(x^\prime,y_3),\eta_\varepsilon \varphi_3(x^\prime,y_3))\in\mathcal{D}(\Omega^-)^3$ satisfying 
$${\rm div}_{x^\prime} \varphi^\prime+\partial_{y_3}\varphi_3=0 \quad \text{in \ }  \Omega^-,\qquad{\rm div}_{x^\prime}\!\left(\int_0^{h_{\rm min}} \varphi^\prime(x^\prime,y_3)  dy_3\right)=0\quad \text{in \ }  \omega.$$
Integrating by parts, we obtain  
\begin{equation*}\label{form_123}
\begin{array}{l}\displaystyle
\mu \int_{\Omega^-}S\left(\mathbb{D}_{\eta_\varepsilon}\left[\tilde{u}'_\varepsilon \right] \right):\mathbb{D}_{\eta_\varepsilon}\left[\varphi^\prime \right]\,dx^\prime dy_3=\int_{\Omega^-}f^\prime\cdot \varphi^\prime\,dx^\prime dy_3 +O_\varepsilon.
\end{array}
\end{equation*}
The procedure to obtain the effective problem is standard and is given in Proposition 3.2 in Mikeli\'c and Tapiero \cite{MT}, so we omit it. Then, we obtain the effective system (\ref{Homogenized_3}). Finally, from Proposition 3.3 in \cite{MT} we have that $\tilde P\in W^{1,p'}(\omega)$.
\end{proof}

In the final step, we will eliminate the microscopic variable $y$ in the effective problem. This is the focus of the Theorem \ref{MainTheorem}. 
\begin{proof}[Proof of Theorem \ref{MainTheorem}]
In the case $\eta_{\varepsilon}\approx \varepsilon$, with $\eta_\varepsilon/\varepsilon\to \lambda$, $0<\lambda<+\infty$ the derivation of (\ref{Main1}) from the effective problem (\ref{effective1}) is  straightforward by using the local problem (\ref{local_3D}) and definition (\ref{def_A_lambda}).

In the case $\eta_{\varepsilon}\ll \varepsilon$, we proceed as the previous case. We deduce that
\begin{equation}
\left\{
\begin{array}
[c]{r@{\;}c@{\;}ll}%
\displaystyle \tilde{V}^\prime(x^{\prime}) &
= &
-\displaystyle {1\over \mu}A^{0}\left(f^\prime(x^{\prime})-\nabla_{x^{\prime}} \tilde{P}(x^{\prime}) \right) \text{\ in \ }\omega,\\
\displaystyle {\rm div}_{x^{\prime}}\, \tilde{V}^\prime(x^{\prime}) & = & 0 \text{\ in \ }\omega,\\
\displaystyle \tilde{V}^\prime(x^{\prime}) \cdot n & = & 0 \text{\ in \ }\partial\omega,
\end{array}
\right. \label{demo}
\end{equation}
where $\tilde V(x^\prime)=\int_0^{h_{\rm max}}\tilde{v}(x^\prime,y_3)\,dy_3$ and $A^{0}:\mathbb{R}^2\to \mathbb{R}^2$ is monotone, coercive and defined by  
\begin{equation}\label{def_A_0_demo}
A^{0}(\xi^\prime)=\int_{Y}w^{\xi^\prime}(y)\,dy,\quad \forall\,\xi^\prime\in\mathbb{R}^2,
\end{equation}
where,  $w^{\xi^\prime}(y^\prime)$, for every $\xi^\prime\in\mathbb{R}^2$, denotes the unique solution in $W^{1,p}_\sharp(Y^\prime)^2$ of the local Stokes problem in 2D
\begin{equation}\label{local_2D_infty_demo}
\left\{\begin{array}{rcl}\displaystyle
-\partial_{y_3}S\left(\partial_{y_3}[w^{\xi^\prime}]\right)+\nabla_{y^\prime}\pi^{\xi^\prime}&=&-\xi^\prime\quad\text{ in \ } Y,\\
\displaystyle
{\rm div}_{y^\prime} \left( \int_0^{h(y')}w^{\xi^\prime}dy_3\right)&=&0\quad\text{ in \ } Y',\\
w^{\xi^\prime}&=&0\quad\text{ on \ } y_3=0,h(y')\\
w^{\xi^\prime}(x',y), \pi^{\xi^\prime}(x',y')\ Y^\prime-\text{periodic}.
\end{array}\right. 
\end{equation}
We observe that (\ref{local_2D_infty_demo}) can be solved, and we can give a Reynolds type equation.

Take into account that
$$\left\vert \partial_{y_3}\left[w^{\xi^\prime} \right]\right\vert^{p-2}=\left\vert Tr\left(\partial_{y_3}\left[w^{\xi^\prime} \right],\partial_{y_3}^t\left[w^{\xi^\prime} \right] \right)\right\vert^{\frac{p}{2}-1},$$
implies
$$S(\partial_{y_3}[w^{\xi^\prime}])=2^{-{p\over 2}}S(\partial_{y_3}w^{\xi^\prime}),$$
from Proposition 3.4 in \cite{MT}, we deduce that
\begin{equation*}
w^{\xi^\prime}(y)=-{2^{p'\over 2}\over p'}\left({h(y')^{p'}\over 2^{p'}}-\left\vert {h(y')\over 2}-y_3\right\vert^{p'} \right)\left\vert \xi^\prime+\nabla_{y'}\pi^{\xi^\prime}\right\vert^{p'-2}\left(\xi^\prime+\nabla_{y'}\pi^{\xi^\prime} \right).
\end{equation*}
From the expression of the Darcy velocity (1.14) in \cite{MT}, we have
\begin{equation*}
\int_0^{h(y')} w^{\xi^\prime}(y)\,dy_3=-{h(y')^{p'+1} \over 2^{p'\over 2}(p'+1)}\left\vert \xi^\prime+\nabla_{y'}\pi^{\xi^\prime}\right\vert^{p'-2}\left(\xi^\prime+\nabla_{y'}\pi^{\xi^\prime} \right).
\end{equation*}
Then, from (\ref{demo})-(\ref{def_A_0_demo}) we have (\ref{Main2}) and (\ref{def_A_0}), and from the second equation in (\ref{local_2D_infty_demo}) we have (\ref{local_2D_infty}).

In the case $\eta_\varepsilon \gg \varepsilon$,  in order to obtain (\ref{thm_Homogenized_3}), we only need to obtain an expression for the velocity $\tilde v'$ in terms of the pressure $\tilde P$ from the first equation in (\ref{Homogenized_3}). This is given in Proposition 3.4 in \cite{MT}, and we have
{\small $$
\tilde v'(x',y_3)={2^{p'\over 2}\over p'\mu^{p'-1}}\!\!\left({h_{\rm min}^{p'}\over 2^{p'}}-\left\vert {h_{\rm min}\over 2}-y_3\right\vert^{p'} \right)\!\!
\left\vert \tilde f'(x')-\nabla_{x'}\tilde P(x')\right\vert^{p'-2}\!\!\left( \tilde f'(x')-\nabla_{x'}\tilde P(x')\right).
$$}
From the expression of the Darcy velocity (1.14) in \cite{MT}, we have (\ref{thm_Homogenized_3}). Finally, from Propositions 3.5  in \cite{MT}, we have that the problem (\ref{thm_Homogenized_3}) has a unique solution.
\end{proof}

\subsection*{Acknowledgments}
Mar\'ia Anguiano has been supported by Junta de Andaluc\'ia (Spain), Proyecto de Excelencia P12-FQM-2466. Francisco J. Su\'arez-Grau has been supported by Ministerio de Econom\'ia y Competitividad (Spain), Proyecto Excelencia MTM2014-53309-P.

\end{document}